\documentclass[a4paper,12pt]{article}

\usepackage{amssymb, amsthm, amsmath} 
\usepackage[dvips]{color}
\usepackage{graphicx}
\usepackage{latexsym}
\usepackage{comment}

\DeclareMathOperator*{\s-lim}{s-lim}
\newcommand{\del}{\partial}

\newtheorem{Def}{\bf Definition}

\newtheorem{Thm}[Def]{\bf Theorem}
\newtheorem{Lem}[Def]{\bf Lemma}

\newtheorem{Prop}[Def]{\bf Proposition}

\renewcommand{\refname}{\bf \normalsize Reference} 

\begin{document}
\title{Scattering theory for Schr\"{o}dinger equations on manifolds with asymptotically polynomially growing ends
}
\author{Shinichiro ITOZAKI\footnote{
Doctoral Course, Graduate School of Mathematical Sciences, The University of Tokyo.
3-8-1 Komaba Meguro-ku Tokyo 153-8914, Japan.
E-mail: randy@ms.u-tokyo.ac.jp.
}}
\maketitle

\begin{abstract}
We study a time-dependent scattering theory for Schr\"{o}dinger operators on a manifold with asymptotically polynomially growing ends. We use the Mourre theory to show the spectral properties of self-adjoint second-order elliptic operators. We prove the existence and the asymptotic completeness of wave operators using the smooth perturbation theory of Kato. We also consider a two-space scattering with a simple reference system.
\end{abstract}

\section{Introduction}\label{Introduction}
We study a class of self-adjoint second-order elliptic operators, which includes Laplacians with long-range potentials on non-compact manifolds which are asymptotically polynomially growing at infinity. We prove Mourre estimate and apply the Mourre theory to these operators. Then we show that there are no accumulation points of embedded eigenvalues except for the zero energy. We obtain resolvent estimates which imply the absence of singular spectrum. We also show the Kato-smoothness for three types of operators.
We construct a time-dependent scattering theory for two operators in our class. If the perturbation is ``short-range'', it admits a factorization into a product of Kato-smooth operators. By virtue of the smooth perturbation theory of Kato, we learn  the existence and the asymptotic completeness of wave operators. Lastly, we consider a two-space scattering with a simple reference system. We follow the settings by Ito and Nakamura.

We now describe our model. Let $M$ be an $n$-dimensional smooth non-compact manifold  such that $M = M_{C}\cup M_{\infty}$, where $M_{C}$ is pre-compact and $M_{\infty}$ is the non-compact end as follows: We assume that $M_{\infty}$ has the form $\mathbb{R}_{+}\times N$ where $N$ is a $n-1$-dimensional compact manifold,  and $\mathbb{R}_{+} = (0, \infty )$ is the real half line. Let $\omega$ be a positive $C^{\infty}$ density $\omega$ on $M$ such that on $M_{\infty}$,
\begin{gather*}
\omega = dr \cdot \mu
\end{gather*}
where $r$ is a coordinate in $\mathbb{R}_{+}$ and $\mu$ is a smooth positive density on $N$. We set $\mathcal{H}=L^{2}(M, \omega)$ be our function space. We set our "free operator" a self-adjoint second-order elliptic operator $L_{0}$ which has the form:
\begin{gather*}
L_{0} = D_{r}^{2} + k(r)P \ \ \ \ \ \ \ \ \ \ \ \text{on} \ (1, \infty) \times N.
\end{gather*}
Here $D_{r} = i^{-1}\del_{r}$, $P$ is a positive self-adjoint second-order elliptic operator acting on $L^{2}(N, \mu)$, and $k$ is a positive smooth function of $r$ such that the derivatives of $k$ satisfy the following estimates for some $c_{0}, C > 0$,
\begin{gather}
c_{0} r^{-1} k\leq -k^{\prime} \leq C r^{-1} k, \label{Conditions_for_k}\\
|k^{\prime \prime} | \leq C r^{-2} k \notag.
\end{gather}
For example $k(r) = r^{-\alpha},$ with $ \alpha > 0$  satisfies the above conditions.

We assume that $L$ is a second-order elliptic operator on $M$, essentially self-adjoint on $C_0^{\infty}(M)$, such that 
\begin{gather*}
L = L_{0} + E,
\end{gather*}
with $E$ having the following properties:
There are finitely many coordinate charts \\
$(r,  \theta_{1}, \cdots, \theta_{n-1})$ on $M_{\infty}$ such that in each chart $E$ has the form
\begin{gather}
E = (1, D_{r}, \sqrt{k}\tilde{D}_{\theta})
\begin{pmatrix}
V &b_{1} & b_{2} \\
b_{1} & a_{1} 	& a_{2} &\\
^{t}b_{2}& ^{t}a_{2} 	& a_{3}	
\end{pmatrix}
\begin{pmatrix}
1\\
D_{r}	\\
\sqrt{k}\tilde{D}_{\theta}
\end{pmatrix} \label{Def_of_E}
\end{gather}
where $\mu(\theta)$ is defined by $\mu = \mu(\theta)d_{\theta_{1}}\cdots d_{\theta_{n-1}}$ and $\tilde{D}_{\theta} = \mu(\theta)^{-\frac{1}{2}}D_{\theta}\mu(\theta )^{\frac{1}{2}}$ is self-adjoint on $L^{2}(N, \mu)$. The coefficients $a_{1}, a_{2}, b_{1}, b_{2},$ and $V$ have support in $M_{\infty}$ and are smooth real-valued functions of $(r,  \theta_{1}, \cdots, \theta_{n-1})$ such that
\begin{gather}
| \del_{r}^{l} \del_{\theta}^{\alpha} a_{j}(r, \theta)| \leq C_{l, \alpha} r^{-\nu _{a_{j}} - l}, \notag \\
| \del_{r}^{l} \del_{\theta}^{\alpha} b_{j}(r, \theta)| \leq C_{l, \alpha} r^{-\nu _{b_{j}} - l},  \label{Perturbation_Hypothesis}\\
| \del_{r}^{l} \del_{\theta}^{\alpha} V(r, \theta)| \leq C_{l, \alpha} r^{-\nu _{V} - l}.  \notag
\end{gather}
Let $\chi(r) \in C^{\infty}(\mathbb{R})$ be a $\mathbb{R}_{+}$-valued  function such that $\chi(r) = 1$ if $r \geq 1$ and $\chi(r) = 0	$ if $r \leq \frac{1}{2}$, and set $\chi_{R}(r) = \chi(\frac{r}{R})$ with $R > 0$.
We set our dilation generator by:
\begin{gather}
A = \frac{1}{2}(\chi_{R}^{2}rD_{r} +D_{r}r\chi_{R}^{2}). \label{Def_of_A}
\end{gather}

Now we state the main results.

\begin{Thm}\label{Thm_Mourre_Theory}
Suppose $L = L_{0} + E$, where $k$ satisfies \eqref{Conditions_for_k} and the coefficients in $E$ obey the bounds \eqref{Perturbation_Hypothesis} with $\nu = \min \{ \nu _{a_{i}}, \nu _{b_{j}}, \nu_{V} \} > 0$. Then $\sigma_{\text{ess}}(L) = \mathbb{R}_{+}\cup \{ 0 \}$ and $L$ satisfies a Mourre estimate at each point in $\mathbb{R}_{+}$ with conjugate operator $A$ in the sense of Definition \ref{Def_of_Mourre_Estimate}. In particular, eigenvalues of $L$ do not accumulate in $\mathbb{R}_{+}$, and $\sigma_{sc}(L)= \emptyset$. We also obtain the resolvent estimates:
\begin{gather*}
\sup_{z \in \Lambda_{\pm} = \Lambda \pm i \mathbb{R}_{+}} 
\| (|A| + 1)^{-s} (L -z)^{-1} (|A| + 1)^{-s} \| < \infty
\end{gather*}
if $\Lambda \Subset \mathbb{R} \setminus \sigma_{\text{pp}}(L)$ and $s > \frac{1}{2}$.
\end{Thm}

We prove Theorem \ref{Thm_Mourre_Theory} in Section \ref{Application of Mourre Theory}.

\begin{Thm}\label{Thm_Kato_Smooth_Operators}
Under the hypotheses of Theorem \ref{Thm_Mourre_Theory}, the operators
\begin{gather*}
G_{0} = \langle r \rangle^{-s},\\
G_{1} = \chi_{R}\langle r \rangle^{-s}D_{r},\\ 
G_{2} = \chi_{R}\langle r \rangle^{-\frac{1}{2}}(kP)^\frac{1}{2}
\end{gather*}
are $L$-smooth on $\Lambda$ if $\Lambda \Subset \mathbb{R} \setminus \sigma_{\text{pp}}(L)$ and $s > \frac{1}{2}$.
\end{Thm}

We prove Theorem \ref{Thm_Kato_Smooth_Operators} in Section \ref{Limiting Absorption Principle}  and Section \ref{Radiation Estimates}.

\begin{Thm}\label{Thm_wave_operator}
Suppose the short-range condition for $E$, that is, $\nu _{a_{1}} = \nu _{a_{2}} = \nu _{b_{1}} = \nu _{b_{1}} = \nu_{V} > 1,$ and $\nu _{a_{3}} =1$.
Then the wave operators
\begin{equation*}
W^{\pm}(L, L_{0}) := \text{s-}\lim_{t\to \pm \infty}e^{itL} e^{-itL_{0}}P_{ac}(L_{0})
\end{equation*}
and $W^{\pm}(L_{0}, L)$ exist and are adjoint each other. They are complete and give the unitarily equivalence between $L_{0}^{(ac)}$ and $L^{(ac)}$.
\end{Thm}

We prove Theorem \ref{Thm_wave_operator} in Section \ref{Wave Operators}.
We note that the wave operators $W^{\pm}(L_{2}, L_{1})$ exist and are asymptotically complete if both of $L_{1}$ and $L_{2}$ satisfy the hypotheses of Theorem \ref{Thm_Mourre_Theory} (long-range ) but the difference $L_{2} - L_{1}$ is short-range in the sense of Theorem \ref{Thm_wave_operator}.

Next we consider a two-space scattering.
We prepare a reference system as follows:
\begin{gather*}
M_{f} = \mathbb{R}\times N,\ \  \mathcal{H}_{f} = L^{2}(M_{f}, H(\theta)dr d\theta ),\\
H_{0} = D_{r}^{2} \ \text{on} \ M_{f},\\
H_{k} = D_{r}^{2} + k(r) P \ \text{on} \ M_{f}.
\end{gather*}
Note that $H_{0}$ and $H_{k}$ are essentially self-adjoint on $C_{0}^{\infty}(M_{f})$,
and we denote the unique self-adjoint extensions by the same symbols.
We define the identification operator $J: \mathcal{H}_{f} \to \mathcal{H}$ by
\begin{equation*}
(Ju)(r,\theta)= \chi(r)u(r, \theta) \ \ \text{if}\ \  (r,\theta) \in M_{\infty}
\end{equation*}
and $Ju(x) = 0$ if $x \notin M_{\infty}$, where $u\in \mathcal{H}_{f}$.
We denote the Fourier transform with respect to $r$ variable by $\mathcal{F}$:
\begin{gather*}
(\mathcal{F} u )(\rho, \theta) = \frac{1}{\sqrt{2 \pi}} \int e^{i r \rho}u(r, \theta ) dr.
\end{gather*}
We set 
\begin{gather*}
\mathcal{H}_{f}^{\pm} : = \mathcal{F}^{-1}[ 1_{\mathbb{R}_{\pm}}(\rho) L^{2}(\mathbb{R} \times N : d\rho\cdot \mu)].
\end{gather*}
Then $\mathcal{H}_{f} = \mathcal{H}_{f}^{+} \oplus \mathcal{H}_{f}^{-}$.

In the two-space scattering, we need additional conditions on $k$:

\begin{Def}\label{Def_of_k_SR_LR}
Suppose that $k$ is a positive smooth function of $r$ satisfying \eqref{Conditions_for_k}.
$k$ is said to be short-range if
\begin{gather}
|k(r)| \leq C \langle r \rangle^{-\nu_{k}} \label{short_range_k}
\end{gather}
with $\nu_{k} > 1$.
$k$ is said to be smooth long-range if
\begin{gather}
| \del_{r}^{l} k(r) | \leq C \langle r \rangle^{-\nu_{k}- l} \label{smooth_long_range_k}
\end{gather}
with $l \in \mathbb{N}$, and $\nu_{k} > 0$.
\end{Def}

For short-range $k$, we have the following.
\begin{Thm}\label{Thm_wave_operator_two_space}
Suppose the hypotheses of Theorem \ref{Thm_wave_operator} and that $k$ is short-range.
Then the wave operators
$W^{\pm}(L, H_{0}; J)$ and $W^{\pm}(H_{0}, L; J^{*})$ exist and are adjoint each other. The asymptotic completeness
\begin{gather*}
W^{\pm}(L, H_{0}; J) \mathcal{H}_{f}^{\pm} = P_{ac}(L) \mathcal{H}
\end{gather*}
holds.
\end{Thm}

For long-range $k$, we need to modify the identifier.
\begin{Thm}\label{Thm_wave_operators_k_long-range_Intro}
Suppose $k$ is smooth long-range in the sense of Definition \ref{Def_of_k_SR_LR}. Fix  $\Lambda \Subset \mathbb{R}$. Then there exists suitable operators $J^{\pm} \in B(\mathcal{H}_{f})$ such that the wave operators $W^{\pm}(L, H_{0}; J J^{\pm})$ and $W^{\pm}(H_{0}, L; (J J^{\pm})^{*})$ exist and are isometric on $E_{\Lambda}(H_{0} )\mathcal{H}_{f}^{\pm}$ and $E_{\Lambda}(L)P_{ac}(L)\mathcal{H}$, respectively, $W^{\pm}(L, H_{0}; J J^{\pm})\mathcal{H}_{f}^{\mp} = 0$, and the asymptotic completeness
\begin{gather*}
W^{\pm}(L, H_{0}; J J^{\pm}) E_{\Lambda}(H_{0} )\mathcal{H}_{f}^{\pm} = E_{\Lambda}(L)P_{ac}(L)\mathcal{H}
\end{gather*}
holds.
\end{Thm}
The construction of modifiers $J^{\pm}$ will be given in Section \ref{Two-space_scattering}.
We can also admit $a_{1}$ to have a long-range part which depends only on $r$. For details, see Section \ref{Two-space_scattering}.

There is a long history on spectral and scattering theory for Schr\"{o}dinger operators (see, for example, \cite{RS72-80}, \cite{Ya00} and references therein). Much of works are connected to differential operators on a Euclidean space. The spectral properties of Laplace operators on a class of non-compact manifolds were studied by Froese, Hislop and Perry \cite{FH89, FHP91}, and Donnelly \cite{D99} using the Mourre theory (see, the original paper Mourre \cite{Mo81}, and Perry, Sigal, and Simon\cite{PSS81}). We follow the settings in Froese and Hislop \cite{FH89}, and Theorem \ref{Thm_Mourre_Theory} may be seen as a direct generalization of \cite{FH89}. We note that only the case with $\nu =  1$ is treated in \cite{FH89}.

In early 1990s, Melrose introduced a new framework of scattering theory on a class of Riemannian manifolds with metrics called scattering metrics (see \cite{M95} and references therein).
He and the other authors have studied Laplace operators on such manifolds. They also studied the absolute scattering matrix, which is defined through the asymptotic expansion of generalized eigenfunctions.

Debi\`{e}vre, Hislop, and Sigal \cite{DHS92} studied a time-dependent scattering theory and proved its asymptotic completeness for some classes of manifolds, including manifolds with asymptotically growing ends with $\nu > 1$.

Ito and Nakamura \cite{IN10} studied a time-dependent scattering theory for Schr\"{o}dinger operators on scattering manifolds. They used  the two-space scattering framework of Kato \cite{Ka67} with a simple reference operator $D_{r}^{2}$ on a space of the form $\mathbb{R} \times N$, where $N$ is the boundary of the scattering manifold $M$.

The case where $M = M_{C}\cup M_{\infty}$ is a Riemannian manifold, the metric on $M_{\infty}$ is "close" to a warped product of $\mathbb{R}_{+}$ and a compact manifold $N$, and $L$ is the Laplace operator, fits into our framework. The function$\sqrt{k(r)}$, varies inversely with the size of $M_{\infty}=\mathbb{R}_{+}\times N$.
A typical exapmle of $k$ which satisfies \eqref{Conditions_for_k} is given by $k(r) = r^{-\alpha}, \alpha > 0$. The case $\alpha = 2$ corresponds to scattering manifolds including asymptotically Euclidean spaces. By using results of Ito and Nakamura \cite{IN10} twice, and by applying the chain rule for wave operators, we can show the existence and the asymptotic completeness of wave operators on scattering manifolds in the one-space scattering framework. Therefore our results can be considered as a generalization of \cite{IN10} for all $\alpha > 0$. In \cite{IN10}, assumptions on $a_{2}$ and $a_{3}$ are weakend to long-range perturbations.  

Our proof of the existence and the asymptotic completeness of wave operators  depends on the smooth perturbation theory of Kato \cite{Ka66} (see also Yafaev \cite{Ya92} and \cite{Ya00}). 
The Kato smoothness of $G_{0} = \langle r \rangle^{-s}$, $s > \frac{1}{2}$ in Theorem \ref{Thm_Kato_Smooth_Operators} is closely related to the limiting absorption principle. The resolvent estimates in the Mourre theory (Theorem \ref{Thm_Mourre_Theory}) imply the limiting absorption principle via a technique in Section 8 of \cite{PSS81}.
The Kato smoothness of $G_{1} = \chi_{R}\langle r \rangle^{-s}D_{r}$ will be obtained in a similar way, but we have to extend the technique in \cite{PSS81} from $\alpha =1$ to $\alpha =2$ (Lemma \ref{Lem_for_LAP2} (i)). 
The Kato-smoothness of $G_{2} = \chi_{R}\langle r \rangle^{-\frac{1}{2}}(kP)^\frac{1}{2}$ is called the radiation estimates. Our proof is quite similar to the one \cite{Ya93}, which relies on the commutator method (see Putnam \cite{Pu67} and Kato \cite{Ka68}).

The limiting absorption principle suffices to show the asymptotic completeness in the case of two-particle Hamiltonians with short-range scalar potentials. However, radiation estimates are crucial in scattering for long-range potentials on a Euclidean space (see Yafaev \cite{Ya00}). In this paper, we found that radiation estimates are also useful for handling short-range metric perturbations and magnetic potentials.
 We hope that we can also construct appropriate modifiers so that the technique of Yafaev can be applied to show the existence and the asymptotic completeness of modified wave operators with long range perturbations in our settings.

In the two-space scattering, essentially we only need to examine wave operators for the pair $(H_{k}, H_{0})$. However, since $P$ commutes with both of $H_{k}$ and $H_{0}$, it reduces to the $1$-dimensional scattering. When $k$ is short-range, we only need a narutal identifier. When $k$ is long-range, we will construct a $1$-dimensional modifiers for the corresponding $1$-dimensional long-range scattering.

\textbf{ACKNOWLEDGMENTS}

I would like to express my thank to my supervisor, Professor Shu Nakamura
for providing me various supports and encouragements.

\section{Application of Mourre Theory}\label{Application of Mourre Theory}

In this section, we prove Theorem \ref{Thm_Mourre_Theory}. 
For the sake of completeness, we give a detailed proof. But methods and proofs used here are almost the same as those in \cite{FH89} and \cite{DHS92}, where $\nu = 2$ and $\nu =1$, respectively, are assumed.
We will prove the Mourre estimate under the condition $\nu > 0$. The index $\nu$ will explicitly appear, for example,  in Lemma \ref{Lem_of_boundedness_of_commutators_of_E} and Lemma \ref{Lem_of_boundedness_of_E}.

We first quote the Mourre theory. 
We define a scale of spaces associated to a self-adjoint operator $L$.
\begin{Def}[Scale of spaces]
Let $L$ be a self-adjoint operator on a Hilbert space $\mathcal{H}$.
For $ s \geq 0$ define $\mathcal{H}_{s} = D((1+|L|)^{\frac{s}{2}}))$ with  the graph norm
\begin{gather*}
\| \psi \|_{s} := \|  (1+|L|)^{\frac{s}{2}}) \psi \|.
\end{gather*}
Define $\mathcal{H}_{-s}$ to be the dual spaces of $\mathcal{H}_{s}$ thought of as the closure of $\mathcal{H}$ in the norm 
\begin{gather*}
\| \psi \|_{-s} := \|  (1+|L|)^{-\frac{s}{2}}) \psi \|.
\end{gather*}
\end{Def}
\begin{Def}[Conjugate Operators]\label{Conjugate_Operators}
Let $L$ be a self-adjoint operator on a Hilbert space $\mathcal{H}$ and $\mathcal{H}_{s}$ be the scale of spaces associated to $L$.
A self adjoint operator $A$ is called a conjugate operator of $L$ if
\begin{enumerate}
\item $D(A) \cap \mathcal{H}_{2}$ is dense in $\mathcal{H}_{2}$,
\item the form $[L, iA]$ defined on $D(A)\cap \mathcal{H}_{2}$ is bounded below and extends to a bounded operator from $\mathcal{H}_{2}$ to $\mathcal{H}_{-1}$,
\item there is a self-adjoint operator $L_{0}$ with $D(L_{0}) = D(L)$ such that $[L_{0}, iA]$ extends to a bounded map from $\mathcal{H}_{2}$ to $\mathcal{H}$, and $D(A) \cap D(L_{0}A)$ is a core for $L_{0}$,
\item the form $[[L, iA], iA]$ extends from $\mathcal{H}_{2}\cap D(LA)$ to a bounded operator from $\mathcal{H}_{2}$ to $\mathcal{H}_{-2}$.
\item $e^{i t A}$ leaves $\mathcal{H}_{2}$ invariant and for each $\psi \in \mathcal{H}_{2}$, $\sup _{|t| \leq 1} \| e^{i t A} \psi \|_{2} < \infty$.
\end{enumerate}
\end{Def}

\begin{Def}[Mourre Estimate]\label{Def_of_Mourre_Estimate}
A self-adjoint operator $L$ satisfies a Mourre estimate on an interval $\Lambda \subset \mathbb{R}$ with conjugate operator $A$ if A is a conjugate operator of $L$ such that 
there exist a positive constant $\alpha$ and a compact operator $K$ such that
\begin{gather*}
E_{\Lambda}[L, iA]E_{\Lambda} \geq \alpha E_{\Lambda} + K.
\end{gather*}
Here $E_{\Lambda}=E_{\Lambda}(L)$ is the spectral projection for $L$.
We say that $L$ satisfies a Mourre estimate at a point $\lambda \in \mathbb{R}$ if
there exists an interval $\Lambda$ containing $\lambda$ such that $L$ satisfies a Mourre estimate on $\Lambda$.
\end{Def}

Now we state the Mourre theory.

\begin{Thm}[Mourre]\label{Thm_General_Mourre_Theory}
Suppose that a self-adjoint operator $L$ satisfies a Mourre estimate at $\lambda \in \mathbb{R}$ with a conjugate operator $A$. Then there exists an open interval $\Lambda$ containing $\lambda$ such that $L$ has finitely many eigenvalues in $\Lambda$ and each eigenvalue has finite multiplicity. If $\lambda \notin \sigma_{pp}(L)$, then there exists an open interval $\Lambda$ containing $\lambda$ such that $L$ has no singular continuous spectrum in $\Lambda$ and for $s > \frac{1}{2}$,
\begin{gather*}
\sup_{z \in \Lambda_{\pm} = \Lambda \pm i \mathbb{R}_{+}} \| (|A| + 1)^{-s} (L -z)^{-1} (|A| + 1)^{-s} \| < \infty.
\end{gather*}
\end{Thm}

We refer to \cite{Mo81} and \cite{PSS81} for the proof of this theorem.

In the following of this section we will show that the hypotheses of the Theorem \ref{Thm_General_Mourre_Theory} will be satisfied for the case stated in section \ref{Introduction}.
\begin{Lem}\label{Lem_Rellich}
Suppose $f \in C_{0}^{\infty}(\mathbb{R})$, D is a differential operator with smooth coefficients, and $\chi$ is a smooth cut-off function with compact support. Then $\chi D f(L)$ and $\chi D f(L_{0})$ are compact operators from $L^{2}(M)$ to $L^{2}(M)$.
\end{Lem}
\begin{proof}
Let $\Omega$ be a bounded domain with smooth boundary which contains supp $\chi$. Then $\chi D f(L_{0})$ and $\chi D f(L)$ map $L^{2}(M)$ to a Sobolev space $H^{s}(\Omega)$ for any $s > 0$.
But $H^{s} \hookrightarrow L^{2}(\Omega) \hookrightarrow L^{2}(M)$, and the first embedding is compact by Rellich's theorem.
\end{proof}

\begin{Lem}\label{Lem_of_boundedness_of differential_operators}
Let $\mathcal{H}_{s}$ be the scale of spaces associated with $L_{0}$.
Then
\begin{enumerate}
\item $\chi_{R} D_{r}: \mathcal{H}_{s} \to \mathcal{H}_{s-1}$ is bounded for $s \in [-1, 2]$, \label{eq_chiRDr}
\item $\chi_{R} D_{r}^{2}: \mathcal{H}_{s} \to \mathcal{H}_{s-2}$ is bounded for $s \in [0, 2]$, \label{eq_chiRDr2}
\item $\chi_{R} ( k P +1)^{\frac{1}{2}}: \mathcal{H}_{s} \to \mathcal{H}_{s-1}$ is bounded for $s \in [-1, 2]$,\label{eq_chiRkPhalf}
\item $\chi_{R} ( k P +1): \mathcal{H}_{s} \to \mathcal{H}_{s-2}$ is bounded for $s \in [0, 2]$.\label{eq_chiRkP}
\end{enumerate}
\end{Lem}
\begin{proof}[Proof of Lemma \ref{Lem_of_boundedness_of differential_operators}]

We begin by proving that
\begin{gather}
\| \chi_{R}D_{r} (L_{0} + C)^{-\frac{1}{2}} \| \leq 1 \label{eq_chiRDr_L2}
\end{gather}
for some constant $C$. Choose a positive constant $C_{1}$ so that $L_{0} + C_{1}$ is a positive operator.
Let $\tilde{\chi}_{R} = (1 - \chi_{R})^{-\frac{1}{2}}$.
The IMS localozation formula gives
\begin{gather*}
L_{0}+C_{1} = \chi_{R}(L_{0}+C_{1})\chi_{R} + \tilde{\chi}_{R}(L_{0}+C_{1})\tilde{\chi}_{R} - (\chi_{R}^{\prime})^{2} - (\tilde{\chi}_{R}^{\prime})^{2}.
\end{gather*}
This implies
\begin{align*}
L_{0}+C_{1} 
& \geq \chi_{R}D_{r}^{2} \chi_{R} - (\chi_{R}^{\prime})^{2} - (\tilde{\chi}_{R}^{\prime})^{2}\\
& \geq D_{r}\chi_{R}^{2}D_{r} -\chi_{R}^{\prime \prime} \chi_{R} -  (\chi_{R}^{\prime})^{2} - (\tilde{\chi}_{R}^{\prime})^{2}
\end{align*}
as form inequalities on $C_{0}^{\infty}$.
Since
\begin{gather*}
|\chi_{R}^{\prime \prime} \chi_{R} -  (\chi_{R}^{\prime})^{2} - (\tilde{\chi}_{R}^{\prime})^{2}|
\leq \frac{C}{R^{2}},
\end{gather*}
this implies 
\begin{gather*}
D_{r} \chi_{R}^{2}D_{r} \leq L_{0} + C.
\end{gather*}
This shows for $\phi \in C_{0}^{\infty}$,
\begin{gather}
\|\chi_{R} D_{r} \phi \| \leq \| (L_{0}+C)^{\frac{1}{2}} \phi \|.\label{eq_chiRDr_C0infinity}
\end{gather}
Since $C_{0}^{\infty}$ is a core for $(L_{0}+C)^{\frac{1}{2}}$, we can find for every $\phi \in D((L_{0}+C)^{\frac{1}{2}})$, a sequence  $\phi_{n} \in C_{0}^{\infty}$ such that $\phi_{n} \to \phi$ and $(L_{0}+C)^{\frac{1}{2}} \phi_{n} \to (L_{0}+C)^{\frac{1}{2}} \phi$. Then
\begin{align*}
|( D_{r}\chi_{R} \psi, \phi)| 
&= \lim_{n \to \infty} |(D_{r} \chi_{R} \psi, \phi)|\\
&= \lim_{n \to \infty} |(\psi, \chi_{R} D_{r} \phi)|\\
&= \limsup_{n \to \infty} \|\psi\| \|(L_{0}+C)^{\frac{1}{2}}\phi_{n}\|\\
&\leq \|\psi\| \|(L_{0}+C)^{\frac{1}{2}}\phi\|
\end{align*}
which shows that $\phi \in D(\chi_{R}D_{r})$ and that \eqref{eq_chiRDr_C0infinity} holds for any $\phi \in D((L_{0} + C)^{\frac{1}{2}})$. Writing $\phi = (L_{0} + C)^{-\frac{1}{2}} \psi$ for $\psi \in L^{2}$, we see that this implies \eqref{eq_chiRDr_L2}.

Next we will prove that
\begin{gather}
\|\chi_{R} D_{r}^{2} (L_{0}+C)^{-1} \| \leq C.\label{eq_chiRDr2_L2}
\end{gather}
With $C_{1}$ as above,
\begin{align*}
(L_{0} + C_{1})^{2}
&\geq (L_{0} + C_{1})\chi_{R}^{2}(L_{0} + C_{1})\\
&= D_{r}^{2} \chi_{R}^{2} D_{r}^{2} + D_{r}^{2} \chi_{R}^{2} (kP + C_{1}) + (kP + C_{1}) \chi_{R}^{2} D_{r}^{2} + (kP + C_{1})^{2} \chi_{R}^{2}\\
&= D_{r}^{2} \chi_{R}^{2} D_{r}^{2} + 2D_{r} \chi_{R}^{2} (kP+C_{1}) D_{r} - (\chi_{R}^{2} (kP+C_{1}) )^{\prime \prime} + ( kP + C_{1} )^{2} \chi_{R}^{2}.
\end{align*}
Using the fact that $|k^{\prime}|$ and $|k^{\prime \prime}|$ are bounded by a constant times $k$, we see that
\begin{gather*}
(\chi_{R}^{2} (kP+C_{1}) )^{\prime \prime} \leq C \chi (kP+C_{1})\chi
\end{gather*}
for some cut-off function $\chi$. Using the IMS formula again, this implies
\begin{gather*}
(\chi_{R}^{2} (kP+C_{1}) )^{\prime \prime} \leq C (L_{0} + C)
\end{gather*}
for some $C$.
Since $2D_{r} \chi_{R}^{2} (kP+C_{1}) D_{r} + ( kP + C_{1} )^{2} \chi_{R}^{2} \geq 0$, we obtain
\begin{gather*}
(L_{0} + C_{1})^{2} \geq D_{r}^{2} \chi_{R}^{2} D_{r}^{2} - C(L_{0} + C),
\end{gather*}
which implies for some $C$,
\begin{gather*}
D_{r}^{2} \chi_{R}^{2} D_{r}^{2} \leq C(L_{0} + C)^{2},
\end{gather*}
which leads to \eqref{eq_chiRDr2_L2} as in the proof of \eqref{eq_chiRDr_L2}.

A similar argument shows that 
\begin{gather}
\|D_{r}^{2} \chi_{R}  (L_{0}+C)^{-1} \| \leq C. \label{eq_Dr2chiR_L2}
\end{gather}
Now by complex interpolation, \eqref{eq_chiRDr2_L2} and \eqref{eq_Dr2chiR_L2}  implies
\begin{gather*}
\|(L_{0}+C)^{-1+z} D_{r}^{2} \chi_{R}  (L_{0}+C)^{-z} \| \leq C.
\end{gather*}
for Re$z$ in $[0, 1]$, which implies \ref{eq_chiRDr2} of Lemma.

To prove \ref{eq_chiRDr} using complex interpolation, one need to prove
\begin{gather}
\| (L_{0} + C)^{-1} \chi_{R} D_{r} (L_{0} + C)^{\frac{1}{2}}\| \leq C, \label{eq_L-1chiRDrLhalf}\\
\| (L_{0} + C)^{\frac{1}{2}} \chi_{R} D_{r} (L_{0} + C)^{-1}\| \leq C. \label{eq_LhalfchiRDrL-1}
\end{gather}
Examining \eqref{eq_L-1chiRDrLhalf}, we see that
\begin{gather*}
\| (L_{0} + C)^{-1} \chi_{R} D_{r} (L_{0} + C)^{\frac{1}{2}}\| \\
\leq \| \chi_{R} D_{r} (L_{0} + C)^{\frac{1}{2}}\| + \| (L_{0} + C)^{-1} [L_{0}, \chi_{R} D_{r}] (L_{0} + C)^{-\frac{1}{2}}\|.
\end{gather*}
The first term on the right hand side is bounded by \eqref{eq_chiRDr_L2}.
The second term can be decomposed into two terms according to the following equation:
\begin{gather*}
[L_{0}, \chi_{R}D_{r}] = [D_{r}^{2}, \chi_{R}D_{r}] + [kP, \chi_{R}D_{r}].
\end{gather*}
The first one is bounded using \ref{eq_chiRDr2} of Lemma;
the second is bounded by
\begin{gather*}
[\chi_{R}, i\chi_{R}D_{r}] = \chi_{R} k P \leq C\chi_{R} k P \leq C (L_{0} + C)
\end{gather*}
and an argument similar to the one above.
This gives \eqref{eq_L-1chiRDrLhalf}.
\eqref{eq_LhalfchiRDrL-1} follows similarly.

\ref{eq_chiRkPhalf} and \ref{eq_chiRkP} follow in a similar way.
\end{proof}

\begin{Lem}\label{Lem_of_boundedness_of_E}
Let $L, L_{0}$ and $E$ be as in Theorem \ref{Thm_Mourre_Theory} and let $\mathcal{H}_{s}$ be the scale of spaces associated with $L_{0}$. Then
\begin{enumerate}
\item $\langle r \rangle^{\nu} E : \mathcal{H}_{s} \to \mathcal{H}_{s-2}$ is bounded for $s \in [0, 2]$,
\item by taking $R$ large enough in the definition of $E$, we may assume that the relative $L_{0}$-bound of $E$ is less than 1,
\item $(L_{0} -z)^{-1} - (L - z)^{-1}$ is compact for $\text{Im}(z) \neq 0$.
\end{enumerate}
\end{Lem}

\begin{proof}[Proof of Lemma\ref{Lem_of_boundedness_of_E}]
(i) will follow by complex interpolation if we can show
\begin{gather*}
\| \langle r \rangle^{\nu} E (L_{0} + i)^{-1} \| \leq C,\\
\| (L_{0} + i)^{-1} \langle r \rangle^{\nu} E \| \leq C.
\end{gather*}
As a typical exapmle, we choose $\sqrt{k}\tilde{D}_{\theta}a_{2}D_{r}$.
Then
\begin{align*}
	&\| \langle r\rangle^{\nu} \chi_{R} \sqrt{k}\tilde{D}_{\theta}a_{2}D_{r}  \chi_{R} (L_{0} + i)^{-1} \| \\
&\leq \| \langle r \rangle^{\nu} \sqrt{k}a_{2} \tilde{D}_{\theta} \chi_{R} D_{r} (L_{0} + i)^{-1} \| 
+ \| \langle r \rangle^{\nu} \sqrt{k} (\tilde{D}_{\theta}a_{2}) \chi_{R} D_{r} (L_{0} + i)^{-1} \|\\
&\leq \sup_{r} \{ \langle r \rangle^{\nu} |a_{2}| \} \cdot \| \sqrt{k}\tilde{D}_{\theta } (kP + 1)^{-\frac{1}{2}} \| \cdot \| \chi_{R} (kP+1)^{\frac{1}{2}} D_{r}(L_{0} + i)^{-1}\|\\
&+ \sup_{r} \{ \langle r \rangle^{\nu} | \tilde{D}_{\theta} a_{2}| \} \cdot \| \sqrt{k} \chi_{R} D_{r} (L_{0} + i)^{-1}\|\\
&\leq C,
\end{align*}	
by assumptions on $a_{2}$ and Lemma \ref {Lem_of_boundedness_of differential_operators}.

To prove (iii), we use the resolvent formula
\begin{gather*}
(L_{0} - z)^{-1} - (L-z)^{-1}\\
= (L-z)^{-1}E(L_{0}-z)^{-1}\\
= (L-z)^{-1}\langle r\rangle^{-\nu} \chi_{R}  \langle r\rangle^{\nu} E(L_{0}-z)^{-1}.
\end{gather*}
The operator $(L-z)^{-1}\langle r\rangle^{-\nu} \chi_{R}$ can be approximated in norm  by operators $f(L)\chi$ considerd in Lemma \ref{Lem_Rellich}, and thus is compact while $\langle r\rangle^{\nu} E(L_{0}-z)^{-1}$ is bounded by (i). This proves (iii).
\end{proof}

\begin{Lem}\label{Lem_Stone_Weierstrass}
$D(L) = D(L_{0})$ and the scale of spaces associated to $L$ and $L_{0}$ are the same. If $f\in C_{0}^{\infty}$, then $f(L) - f(L_{0})$ is compact.
\end{Lem}
\begin{proof}[Proof of Lemma \ref{Lem_Stone_Weierstrass}]
The first statement is obtained by the relative boundedness, and the second follows from (iii) of Lemma \ref{Lem_of_boundedness_of_E}  and a Stone-Weierstrass argument.
\end{proof}

\begin{Lem}
$\sigma_{ess}(L) = [0, \infty)$.
\end{Lem}
\begin{proof}
The Persson's formula (see, for exapmle, \cite{CFKS86})
\begin{gather*}
\inf \sigma_{\text{ess}}(L) = \sup_{ K \Subset M } \inf_{ \phi \in C_{0}^{\infty}(M \setminus K), \| \phi \| = 1 } \langle \phi,  L \phi \rangle
\end{gather*}
and a Weyl sequene argument give the desired result. 
\end{proof}

\begin{Lem}\label{Lem_boundedness_of_free_commutator}
Let $L_{0}$ be as in Theorem \ref{Thm_Mourre_Theory}. Then for large enough $R$,
\begin{enumerate}
\item $[L_{0}, iA]$ extends from $C_{0}^{\infty}$ to a bounded operator $\mathcal{H}_{+2} \to \mathcal{H}$,
\item $[[L_{0}, iA], iA]$ extends from $C_{0}^{\infty}$ to a bounded operator $\mathcal{H}_{+2} \to \mathcal{H}_{-2}$.
\end{enumerate}
\end{Lem}
\begin{proof}
To begin, we show that $[D_{r}^{2}, iA]$ is bounded from $\mathcal{H}_{+2}$ to $\mathcal{H}$.
A brief calculation shows
\begin{gather*}
[D_{r}^{2}, iA] = 2 (\chi_{R}^{2} r )^{\prime} D_{r}^{2} + \frac{2}{i}(\chi_{R}^{2} r)^{\prime \prime}D_{r} - \frac{1}{2}(\chi_{R}^{2} r) ^{\prime \prime \prime}.
\end{gather*}
The coefficients $(\chi_{R}^{2} r )^{\prime}, (\chi_{R}^{2} r)^{\prime \prime},$ and $ (\chi_{R}^{2} r) ^{\prime \prime \prime}$ are bounded.
By taking $R$ large enough, the boundedness of $[D_{r}^{2}, iA]$ from $\mathcal{H}_{+2}$ to $\mathcal{H}$ follows from that of $\chi_{R}D_{r}^{2}$ and $\chi_{R}D_{r}$, which is ensured by  Lemma \ref{Lem_of_boundedness_of differential_operators}.

Next we consider the term
\begin{gather*}
[kP, iA] = -\chi_{R}^{2} r k^{\prime} P.
\end{gather*}
By \eqref{Conditions_for_k}, $|r k^{\prime} | \leq C k$.
Using Lemma \ref{Lem_of_boundedness_of differential_operators}, it follows that $[kP, iA]$ is bounded from $\mathcal{H}_{+2}$ to $\mathcal{H}$. This completes the proof of (i).

The boundedness of the double commutator in (ii) is proven using similar arguments.
We can use Lemma \ref{Lem_of_boundedness_of differential_operators} to prove the boundedness of $[ [D_{r}^{2}, iA], iA]$ from $\mathcal{H}_{+2}$ to $\mathcal{H}_{-2}$.
Since
\begin{gather*}
[ [kP, iA], iA] = \chi_{R}^{2}r(\chi_{R}^{2}r k^{\prime})^{\prime} P,
\end{gather*}
we need the estimates \eqref{Conditions_for_k} on the second derivative of $k$ for $r$ large
\begin{gather*}
|r^{2} k^{\prime \prime}| \leq C k
\end{gather*}
to prove the boundedness of $[ [kP, iA], iA]$.
\end{proof}

Now we prove the Mourre estimate for unperturbated system $L_{0}$.

\begin{Lem}
Let $L_{0}$ be as in Theorem \ref{Thm_Mourre_Theory} and $A$ given by \eqref{Def_of_A}. Suppose $\lambda_{0} > 0$. Then for every $\epsilon > 0$ there exist an interval $\Lambda$ about $\lambda_{0}$ and a compact operator $K$ such that for $R$ large,
\begin{gather*}
E_{\Lambda}(L_{0}) [L_{0}, iA] E_{\Lambda}(L_{0}) \geq \min(2, c_{0})(\lambda_{0} - \epsilon )E_{\Lambda}(L_{0}) + K.
\end{gather*}
Here $E_{\Lambda}(L_{0})$ is the spectral projection for $L_{0}$ corresponding to $\Lambda$, and 
 $c_{0}$ is the constant which appears in \eqref{Conditions_for_k}.
\end{Lem}
\begin{proof}
Choosing $R$ large, we have
\begin{gather*}
[D_{r}^{2}, iA] = 2D_{r} (\chi_{R}^{2} r)^{\prime} D_{r} - \frac{1}{2}(\chi_{R}^{2} r)^{\prime \prime \prime}\\
\geq 2 D_{r} \chi_{R}^{2} D_{r} - \frac{\epsilon}{4} \min\{ 2, c_{0} \} \\
\geq 2 \chi_{R} D_{r}^{2} \chi_{R}  - \frac{\epsilon}{2} \min\{ 2, c_{0} \}.
\end{gather*}
Also, 
\begin{gather*}
[kP, iA] = - \chi_{R}^{2} r k^{\prime} P \geq c_{0} \chi_{R}^{2} k P.
\end{gather*}
Combining these two inequalities, we obtain
\begin{align*}
[L_{0}, iA]
&\geq \chi_{R} (2D_{r}^{2} + c_{0}kP) \chi_{R} - \frac{\epsilon}{2} \min\{ 2, c_{0} \} \\
&\geq \min\{ 2, c_{0} \} (\chi_{R} L_{0} \chi_{R} - \frac{\epsilon}{2}).
\end{align*}
We now multiply this estimate on both sides with $f(L_{0})$ where $f$ is a smooth compactly supported characteristic function of an interval about $\lambda_{0}$. This gives
\begin{align*}
f(L_{0}) [L_{0}, iA] f(L_{0}) \geq 
\min\{ 2, c_{0} \} ( f(L_{0}) \chi_{R} L_{0} \chi_{R}f(L_{0}) - \frac{\epsilon}{2} f^{2}(L_{0})) .
\end{align*}
Now 
\begin{align*}
&f(L_{0}) \chi_{R} L_{0} \chi_{R}f(L_{0})\\
=  &f(L_{0}) L_{0} (\chi_{R} - 1 )f(L_{0}) +  f(L_{0}) (\chi_{R} - 1) L_{0} \chi_{R}f(L_{0})
+  f(L_{0}) L_{0} f(L_{0}).
\end{align*}
$f \in C_{0}^{\infty}$ implies $f(L_{0}) L_{0}$ is bounded.
It is not difficult to see that $L_{0} \chi_{R}f(L_{0})$ is bounded using Lemma \ref{Lem_of_boundedness_of differential_operators}.
Since $\chi_{R} - 1$ has compact support, $(\chi_{R} - 1 )f(L_{0})$ is compact by Lemma \ref{Lem_Rellich}.
Thus, if the support of $f$ is within $\frac{\epsilon}{2}$ of $\lambda_{0}$, we have
\begin{gather*}
f(L_{0}) \chi_{R} L_{0} \chi_{R}f(L_{0}) \geq (\lambda_{0} - \frac{\epsilon}{2} ) f^{2}(L_{0}) + K.
\end{gather*}
where $K$ is a compact operator.
Therefore we have
\begin{gather}
f(L_{0}) [L_{0}, iA] f(L_{0}) \geq 
\min \{ 2, c_{0} \} (\lambda_{0} - \epsilon ) f^{2}(L_{0}) + K. \label{smooth_free_Mourre_estimate}
\end{gather}
Taking $f = 1$ in a neighbourhood of $\lambda_{0}$ and multiplying from both sides with $E_{\Lambda}(L_{0})$, with $\Lambda$ small enough to ensure  $E_{\Lambda}(L_{0})f(L_{0}) =  E_{\Lambda}(L_{0})$, this inequality gives the desired Mourre estimate.

\end{proof}

\begin{Lem}\label{Lem_of_boundedness_of_commutators_of_E}
Under the hypotheses of Theorem \ref{Thm_Mourre_Theory},
\begin{enumerate}
\item $[E, iA]: \mathcal{H}_{+2} \to \mathcal{H}_{0}$ is bounded,
\item $f(L)[E, iA]f(L)$: is compact for $f \in C_{0}^{\infty}$,
\item $[[E, iA], iA]: \mathcal{H}_{+2} \to \mathcal{H}_{-2}$ is bounded.
\end{enumerate}
\end{Lem}
\begin{proof}
It is easy to see that $[E, iA]$ has the following form:
\begin{gather}
[E, iA] = 
(1, D_{r}, \sqrt{k}\tilde{D}_{\theta})
\chi_{R}
\begin{pmatrix}
\tilde{V} &\tilde{b}_{1} & \tilde{b}_{2} \\
\tilde{b}_{1} & \tilde{a}_{1} 	& \tilde{a}_{2} &\\
^{t}\tilde{b}_{2}& ^{t}\tilde{a}_{2} 	& \tilde{a}_{3}	
\end{pmatrix}
\chi_{R}
\begin{pmatrix}
1\\
D_{r}	\\
\sqrt{k}\tilde{D}_{\theta}
\end{pmatrix} \label{[E,iA]}
\end{gather}
where
$\tilde{e} = \tilde{a}_{1}, \tilde{a}_{2}, \tilde{b}_{1}, \tilde{b}_{1}$ and $\tilde{V}$ satisfy 
\begin{gather}
|\tilde{e}(r, \theta)| \leq C r^{-\nu},\ \ \nu > 0.\label{estimates_on_e_tilde}
\end{gather}
Here we used the estimates \eqref{Perturbation_Hypothesis} on the first derivatives with respect to $r$ of coefficients in $E$ and the estimates \eqref{Conditions_for_k} on the first derivative of $k$.
By Lemma \ref{Lem_of_boundedness_of differential_operators}, \eqref{[E,iA]} and \eqref{estimates_on_e_tilde} imply $[E, iA]$ is bounded from $\mathcal{H}_{+2} \to \mathcal{H}_{0}$, which is (i). 

The boundedness of the double commutator in (iii) is proven using similar arguments. We need the estimates on second derivatives with respect to $r$ of coefficients in $E$ and $k$.

To prove (ii), note that $\chi_{R}(1, D_{r}, \sqrt{k}\tilde{D}_{\theta}) f(L)$ is bounded
and $\chi_{R}r^{-\nu}(1, D_{r}, \sqrt{k}\tilde{D}_{\theta}) f(L)$ is compact by Lemma \ref{Lem_Rellich}. Hence \eqref{[E,iA]} implies (ii).
\end{proof}

\begin{proof}[Proof of Theorem \ref{Thm_Mourre_Theory}]
We first show that $L$ and $A$ satisfy the conditions in Definition \ref{Conjugate_Operators}, that is, $A$ is a conjugate operator of $L$.
Since $C_{0}^{\infty} \subset D(A) \cap \mathcal{H}_{2}$ is a core for $L$ by hypothesis, condition (i) in Definition \ref{Conjugate_Operators} is satisfied.
Condition (ii) follows from (i) of Lemma \ref{Lem_boundedness_of_free_commutator} and (i) of Lemma \ref{Lem_of_boundedness_of_commutators_of_E}.
The first statement of (iii) follows from Lemma \ref{Lem_Stone_Weierstrass} and (i) of Lemma \ref{Lem_boundedness_of_free_commutator}. The second statement follows from the inclusion $C_{0}^{\infty} \subset D(A)\cap D(L_{0}A)$.
Condition (iv) follows from (ii) of Lemma \ref{Lem_boundedness_of_free_commutator} and (iii) of Lemma \ref{Lem_of_boundedness_of_commutators_of_E}. 
Let $X$ be a vector field on $M$ such that
\begin{gather*}
X = r \chi_{R}^{2}(r) \frac{\del}{\del r} \ \ \text{on} M_{\infty},
\end{gather*}
and $X = 0$ on $M_{C}$. Let $\{ \exp[tX] | t \in \mathbb{R} \} $ be the flow generated by X. The flow induces a one-parameter unitary group defined by
\begin{gather*}
U(t)\phi(x) = \Phi(t, x) \phi(\exp[-t X] x) 
\end{gather*}
for $\phi \in \mathcal{H}$, where $\Phi(t, x)$ is a weight function to make the dilation operator $U(t)$ unitary. By simple calculation, we find that
\begin{gather*}
A = \frac{1}{2}(\chi_{R}^{2}rD_{r} +D_{r}r\chi_{R}^{2})
\end{gather*}
is the generator of the dilation operator $U(t)$, that is, $U(t) = e^{-i t A}$. Now it is easy to see $e^{-i t A}$ leaves $D(L) = \mathcal{H}_{2}$ invariant and to show (v) as in the Euclidean case.

Now we show the Mourre estimate.
We replce $L_{0}$ with $L$ in \eqref{smooth_free_Mourre_estimate}.
By Lemma \ref{Lem_Stone_Weierstrass}, $f^{2}(L) - f^{2}(L_{0})$ and $f(L) - f(L_{0})$ are compact. By Lemma \ref{Lem_boundedness_of_free_commutator}, we can see that $[L_{0}, iA]f(L)$ and $f(L_{0})[L_{0}, iA]$ are bounded.
$f(L)[E, iA]f(L)$ is compact by (ii) of Lemma \ref{Lem_of_boundedness_of_commutators_of_E}. Using these facts, it is easily seen that replacing $L_{0}$ with $L$ in \eqref{smooth_free_Mourre_estimate} introduces a compact error, which can be handled in $K$. Making this replacement and multiplying the resulting equation from both sides with $E_{\Lambda}(L_{0})$, with $\Lambda$ small enough to ensure  $E_{\Lambda}(L_{0})f(L_{0}) =  E_{\Lambda}(L_{0})$, give the Mourre estimate
\begin{gather*}
E_{\Lambda}(L) [L, iA] E_{\Lambda}(L) \geq \min(2, c_{0})(\lambda_{0} - \epsilon )E_{\Lambda}(L) + K.
\end{gather*}
We have showed that $L$ satisfies a Mourre estimate at any point $\lambda_{0} > 0$ with conjugate operator $A$, which completes the proof of Theorem \ref{Thm_Mourre_Theory}.
\end{proof}

\section{Limiting Absorption Principle}\label{Limiting Absorption Principle}
In this section we show the limiting absorption principle, which leads to the Kato-smoothness of $G_{0}$ and $G_{1}$ in Theorem \ref{Thm_Kato_Smooth_Operators}. We extend the discussion in \cite{PSS81}.

We will prove

\begin{Thm}\label{Thm_Limiting_Absorption_Principle}
Let $L$ be as in Theorem \ref{Thm_Mourre_Theory}, $s > \frac{1}{2}$, and $\Lambda \Subset \mathbb{R}_{+} \setminus \sigma_{pp}(L)$. Then
\begin{gather*}
\sup_{z \in \Lambda_{\pm} = \Lambda \pm i \mathbb{R}_{+}} 
\| (|r| + 1)^{-s} (L -z)^{-1} (|r| + 1)^{-s} \| < \infty \\
\sup_{z \in \Lambda_{\pm} = \Lambda \pm i \mathbb{R}_{+}} 
\| (|r| + 1)^{-1-s} A(L -z)^{-1}A (|r| + 1)^{-1-s} \| < \infty.
\end{gather*}
\end{Thm}
The first estimate is called the limiting absorption principle.

As a preliminary we prove

\begin{Lem}\label{Lem_for_LAP1}
Let $L$ be as in Theorem \ref{Thm_Mourre_Theory}. Then
\begin{enumerate}
\item $[E, ir\chi_{R}](L+i)^{-1}$ is bounded,
\item $[[E, ir\chi_{R}], ir\chi_{R}](L+i)^{-1}$ is bounded,
\item $\chi_{R}D_{r}r\chi_{R}(L+i)^{-1}\langle r \rangle^{-1}$ is bounded,
\item $\chi_{R}D_{r}^{2}r^{2}\chi_{R}(L+i)^{-1}\langle r \rangle^{-2}$ is bounded.
\end{enumerate}
\end{Lem}
\begin{proof}


By \eqref{Conditions_for_k}, \eqref{Perturbation_Hypothesis} and
 Lemma \ref{Lem_of_boundedness_of differential_operators}, we can show (i) and (ii).

Now we compute (iii).
\begin{align*}
& \chi_{R}D_{r}r\chi_{R}(L+i)^{-1}\langle r \rangle^{-1}\\
=& \chi_{R}D_{r}(L+i)^{-1}r\chi_{R}\langle r \rangle^{-1} 
+ \chi_{R}D_{r}(L+i)^{-1} [L, r\chi_{R}] (L+i)^{-1} \langle r \rangle^{-1}
\end{align*}
The first term in the right hand side is bounded by (i) of Lemma \ref{Lem_of_boundedness_of differential_operators}. We have
\begin{align*}
[L, r\chi_{R}] &= [D_{r}^{2}, r\chi_{R}] + [E, r\chi_{R}] \\
&= 2i^{-1}(r\chi_{R})^{\prime}D_{r} - (r\chi_{R})^{\prime \prime} + [E, r\chi_{R}].
\end{align*}
Using  (i) of Lemma \ref{Lem_of_boundedness_of differential_operators}  and (i) of Lemma \ref{Lem_for_LAP1}, we obtain the boundedness of $[L, r\chi_{R}] (L+i)^{-1}$, which implies the boundedness of the second term.

Next we will show (iv). we begin with the equality
\begin{align*}
\chi_{R}D_{r}^{2}r^{2}\chi_{R}(L+i)^{-1}\langle r \rangle^{-2}
= \chi_{R}D_{r}^{2} (L+i)^{-1} r^{2}\chi_{R}\langle r \rangle^{-2}
+ \chi_{R}D_{r}^{2} (L+i)^{-1} [L, r^{2}\chi_{R}] (L+ i)^{-1} \langle r \rangle^{-2}.
\end{align*}
The first term in the right hand side is bounded by (ii) of Lemma \ref{Lem_of_boundedness_of differential_operators}. The second term can be decomposed into two terms according to the following:
\begin{align*}
[L, r^{2}\chi_{R}] &= [D_{r}^{2}, r^{2}\chi_{R}] + [E, r^{2}\chi_{R}].
\end{align*}
It is easy to see that the first one is bounded using (iii). Replacing $\chi_{R}$ by $\chi_{R}^{2}$, the second can be decomposed in the following way:
\begin{align*}
&[E, r^{2}\chi_{R}^{2}](L+i)^{-1}\langle r \rangle^{-1} \\
=& 2[E, r\chi_{R}]r\chi_{R}(L+i)^{-1} \langle r \rangle^{-1}+ [r\chi_{R}, [E, r\chi_{R}]](L+i)^{-1}\langle r \rangle^{-1}\\
=& 2[E, r\chi_{R}](L+i)^{-1} r\chi_{R}\langle r \rangle^{-1}
+ 2[E, r\chi_{R}](L+i)^{-1} [L, r\chi_{R}] (L+i)^{-1}\langle r \rangle^{-1}\\
+& [r\chi_{R}, [E, r\chi_{R}]](L+i)^{-1}\langle r \rangle^{-1},
\end{align*}
which is bounded using (i), (ii) and the boundedness of $[D_{r}^{2}, r\chi_{R}](L+i)^{-1}\langle r \rangle^{-1}$, which can be shown by the argument in (iii). This proves (iv).
\end{proof}

\begin{Lem}\label{Lem_for_LAP2}
Let L be as in Theorem \ref{Thm_Mourre_Theory}.
Then
\begin{enumerate}
\item $\langle | A |\rangle^{\alpha} (L+i)^{-1} \langle r \rangle^{-\alpha}$ is bounded for $0 \leq \alpha \leq 2$
\item $\langle | A |\rangle^{s} [ A, (L+i)^{-1}] \langle r \rangle^{-1-s}$ is bounded for $0 \leq s \leq 1$.
\end{enumerate}
\end{Lem}
\begin{proof}[Proof of Lemma \ref{Lem_for_LAP2}]
By interpolation, it is enough to prove for $\alpha =0, 2$ and $s = 0, 1$.
The case $\alpha = 0$ is obvious. The case $\alpha =2$ and $s = 1$ follows from (iii) and (iv) of Lemma \ref{Lem_for_LAP1}.

The case $s = 0$ follows from Lemma \ref{Lem_boundedness_of_free_commutator} and Lemma \ref{Lem_of_boundedness_of_commutators_of_E}.
\end{proof}

\begin{proof}[Proof of Theorem \ref{Thm_Limiting_Absorption_Principle}]
Writing
\begin{gather*}
\langle r \rangle^{-s} (L+i)^{-1} (L-z)^{-1}(L+i)^{-1} \langle r \rangle^{-s}\\
= \langle r \rangle^{-s} (L+i)^{-1} \langle | A |\rangle^{-s} \cdot \langle | A |\rangle^{s} (L-z)^{-1} \langle | A |\rangle^{s} \cdot \langle | A |\rangle^{-s} (L+i)^{-1} \langle r \rangle^{-s}
\end{gather*}
and using Theorem \ref{Thm_Mourre_Theory} and Lemma \ref{Lem_for_LAP2},
we see that
\begin{gather*}
\langle r \rangle^{-s} (L+i)^{-1} (L-z)^{-1} (L+i)^{-1} \langle r \rangle^{-s}
\end{gather*}
is bounded.

Also, writing
\begin{gather*}
\langle r \rangle^{-1-s} A (L+i)^{-1} (L-z)^{-1}(L+i)^{-1} A \langle r \rangle^{-1-s}\\
= \langle r \rangle^{-1-s} A (L+i)^{-1} \langle | A |\rangle^{-s} \cdot \langle | A |\rangle^{s} (L-z)^{-1} \langle | A |\rangle^{s} \cdot \langle | A |\rangle^{-s} (L+i)^{-1} A \langle r \rangle^{-1-s}
\end{gather*}
and using Theorem \ref{Thm_Mourre_Theory} and Lemma \ref{Lem_for_LAP2},
we see that
\begin{gather*}
\langle r \rangle^{-1-s} A (L+i)^{-1} (L-z)^{-1} (L+i)^{-1} A \langle r \rangle^{-1-s}
\end{gather*}
is bounded.

Since
\begin{gather*}
(L-z)^{-1} = (L+i)^{-1} + (z+i)(L+i)^{-2} +(z+i)^{2} (L+i)^{-1} (L-z)^{-1} (L+i)^{-1},
\end{gather*}
we obtain the desired result.
\end{proof}

\section{Radiation Estimates}\label{Radiation Estimates}
In this section, we prove the radiation estimates. 
We want first to recall general definitions of Kato-smoothness and the commutator method which allow us to find new Kato-smooth operators $K$ given Kato-smooth operators $G$. For details, we refer the textbooks Yafaev \cite{Ya92} and \cite{Ya00}.

\begin{Def}\label{Def_of_Kato-smoothness}
An $H$-bounded operator $G$ is called $H$-smooth in the sense of Kato if
\begin{align*}
& \sup_{f \in D(H), \| f \| = 1} \int_{-\infty}^{\infty}\| G e^{-iHt}f\|^{2} dt \\
= &\sup_{z \in \mathbb{R}+i\mathbb{R}}\|G ((H - z)^{-1} - (H - \bar{z})^{-1} )G^{*} \| \\
< &\infty.
\end{align*}
An operator $G$ is called $H$-smooth on a Borel set $\Lambda$ if $GE(\Lambda)$ is $H$-smooth, which is equivalent to the condition
\begin{gather*}
\sup_{z \in \Lambda+i\mathbb{R}}\|G ((H - z)^{-1} - (H - \bar{z})^{-1} )G^{*} \| < \infty,
\end{gather*}
where $E(\Lambda)$ is the spectral projection of $H$ on $\Lambda$.
\end{Def}

\begin{Prop}\label{Copmmutator_method_to_find_new_Kato_smooth_operators}
Suppose that 
\begin{gather*}
G^{*}G \leq i[H, M] + K^{*}K,
\end{gather*}
where $M$ is a $H$-bounded operator and $K$ is $H$-smooth on a Borel set $\Lambda$. Then $G$ is also $H$-smooth on $\Lambda$.
\end{Prop}
For the proof of Proposition \ref{Copmmutator_method_to_find_new_Kato_smooth_operators}, see Proposition 1.19 in \cite{Ya00}.

Now we return to our problem.

\begin{Thm}\label{Thm_Radiation_Estimates}
Let $L$ be as in Theorem \ref{Thm_Mourre_Theory}. Then for large enough $R$,
\begin{gather*}
\chi_{R} r^{-\frac{1}{2}} (kP)^{\frac{1}{2}} 
\end{gather*}
is $L$-smooth on $\Lambda$ if $\Lambda \Subset \mathbb{R} \setminus \sigma_{\text{pp}}(L)$.
\end{Thm}

We prepare the following lemma.
\begin{Lem}\label{Lemma_for_Radiation_Estimates}
For every $\epsilon > 0$, there exist a constant $C > 0$ such that
\begin{gather}
(c_{0} - \epsilon ) G_{2}^{*}G_{2} \leq [L, iM] + C \sum_{j, k = 0, 1} G_{j}^{*} G_{j}\label{Estimete_for_Lemma_for_Radiation_Estimates}
\end{gather}
where 
\begin{gather*}
M = \frac{1}{2}(\chi_{R} D_{r} + D_{r} \chi_{R})\\
G_{0} = \langle r \rangle^{-s},\\
G_{1} = \chi_{R}\langle r \rangle^{-s}D_{r},\\ 
G_{2} = \chi_{R}\langle r \rangle^{-\frac{1}{2}}(kP)^\frac{1}{2}\\
s = \frac{1}{2}(1 + \nu) > \frac{1}{2}
\end{gather*}
and $c_{0}$ is the constant which appears in \eqref{Conditions_for_k}.
\end{Lem}

\begin{proof}[Proof of Lemma \ref{Lemma_for_Radiation_Estimates}.\ \ ]
To calculate the commutator $[L, iM]$, we first remark that
\begin{gather}
[D_{r}^{2}, iM] = 2 D_{r}\chi_{R}^{\prime} D_{r}\label{Estimate_of_[D_{r}^{2}, iM]}\\
[k(r)P, iM] = - \chi_{R} k^{\prime} P \geq c_{0} \chi_{R} r^{-1} k P. \label{Estimate_of_[k(r)P, iM]}
\end{gather}
Here we used the inequality \eqref{Conditions_for_k}.

For the perturbation term $[E, iM]$. we can prove
\begin{gather}
|([E, iM]u, u)| \leq C \|G_{0} u\|^{2} + \|G_{1} u\| + C \|\chi_{R}r^{-\nu} \| \|G_{2}u\|^{2}.\label{Estimate_of_[E,iM]}
\end{gather}
It suffices to prove this estimate for each term of $E$ in the sum \eqref{Def_of_E}.
First we consider the terms involving $V$.
\begin{gather*}
[V, iM] = -\chi_{R} V^{\prime},\\
|([V, iM]u, u)| \leq C \|\chi_{R}r^{-\frac{1+\nu}{2}} u\|^{2} 
\leq C \|G_{0} u\|^{2}.
\end{gather*}
For the $a_{1}$ part, we have that
\begin{gather*}
[D_{r} a_{1} D_{r}, iM] = -D_{r}(\chi_{R}a_{1})^{\prime}D_{r},\\
|([D_{r} a_{1} D_{r}, iM]u, u)| \leq C \|G_{1}u\|^{2}.
\end{gather*}
For the $a_{3}$ part, we have that
\begin{gather*}
[\tilde{D}_{\theta} k a_{3} \tilde{D}_{\theta}, iM] = -\tilde{D}_{\theta} \chi_{R} (ka_{3})^{\prime} \tilde{D}_{\theta}\\
|([\tilde{D}_{\theta} k a_{3} \tilde{D}_{\theta}, iM] u, u)| 
\leq C \| \chi_{R} r^{- \nu}\| \cdot   \| \chi_{R} r^{-\frac{1}{2}}(k P)^{\frac{1}{2}} u \|^{2}
= C \| \chi_{R} r^{- \nu}\| \|G_{2} u\|^{2}.
\end{gather*}
Other terms can be handled in a similar way.

Combinig the inequalities \eqref{Estimate_of_[D_{r}^{2}, iM]}, \eqref{Estimate_of_[k(r)P, iM]} and \eqref{Estimate_of_[E,iM]}, we arrive at the estimate
\begin{gather*}
([L, iM]u, u) \geq c_{0} \| G_{2} u\|^{2}  - C \|G_{0}u\|^{2} - C \|G_{1}u\|^{2} - \| \chi_{R} r^{- \nu}\| \|G_{2} u\|^{2}\\
\geq (c_{0}- \epsilon ) \| G_{2} u\|^{2} - C \|G_{0}u\|^{2} - C \|G_{1}u\|^{2}
\end{gather*}
for an arbitrary $\epsilon > 0$ by taking $R > 0 $ large enough.
This gives the desired estimate \eqref{Estimete_for_Lemma_for_Radiation_Estimates}.
\end{proof}

\begin{proof}[Proof of Theorem \ref{Thm_Radiation_Estimates}]
Fix $\Lambda \Subset \mathbb{R} \setminus \sigma_{\text{pp}}(L)$ and consider \eqref{Estimete_for_Lemma_for_Radiation_Estimates}.
The operators $G_{0}$ and $G_{1}$ are $L$-smooth on $\Lambda$ by Theorem \ref{Thm_Limiting_Absorption_Principle} and $G_{2}$ is $L$-bounded.
The commutator method Proposition \ref{Copmmutator_method_to_find_new_Kato_smooth_operators} implies that the operator $G_{2}$ is also $L$-smooth on $\Lambda$.
\end{proof}

Theorem \ref{Thm_Limiting_Absorption_Principle} and Theorem \ref{Thm_Radiation_Estimates} directly mean Theorem\ref{Thm_Kato_Smooth_Operators}.

\section{One-space scattering}\label{Wave Operators}
We recall the smooth method of Kato which assures the existence of wave operators for perturbations that are smooth locally. For more details, see Corollary 4.5.7. in \cite{Ya92}.
\begin{Thm}\label{Kato's_smooth_method}
Suppose that $H$ and $H_{0}$ are self-adjoint operators  on Hilbert spaces $\mathcal{H}$ and $\mathcal{H}_{0}$ respectively, $J \in B(\mathcal{H}_{0}, \mathcal{H})$ is the identifier, and the pertuabation $HJ - JH_{0}$ admits a factorization
\begin{gather*}
HJ - JH_{0} = G^{*} G_{0},
\end{gather*}
where $G_{0}$ is $H_{0}$-bounded and $G$ is $H$-bounded.
Suppose $\{ \Lambda_{n} \}$ is a set of intervals which exhausts the core of the spectra of the operators $H_{0}$ and $H$ up to a set of Lebesgue measure zero. If on each of the intervals $\Lambda_{n}$ the operator $G_{0}$ is $H_{0}$-smooth and $G$ is $H$-smooth, then the wave operators $W^{\pm}(H, H_{0}; J)$ and $W^{\pm}(H_{0}, H; J^{*})$ exist.
\end{Thm}

Now we apply Theorem \ref{Kato's_smooth_method} to our model.
\begin{proof}[Proof of Theorem\ref{Thm_wave_operator}]
First we note that any first order differential operator with compactly supported smooth coefficient function is $L_{0}$- and $L-$ locally smooth. This fact can be easily proved as in Section\ref{Limiting Absorption Principle}.

The perturbation term $E$ admits a factorization of the following form
\begin{gather*}
E = \sum_{l, m = 0, 1, 2}G_{l}^{*}B_{l,m}G_{m} + E_{C}
\end{gather*}
where $G_{l}$ are $L_{0}$-smooth on any $\Lambda \Subset \mathbb{R} \setminus \sigma_{\text{pp}}(L_{0})$ and $L$-smooth on any $\Lambda \Subset \mathbb{R} \setminus \sigma_{\text{pp}}(L)$ and $E_{C}$ is a second-order differential operator with compactly supported coefficient function. Then the smooth perturbation theory of Kato shows the existence of the wave operators $W^{\pm}(L, L_{0})$ and $W^{\pm}(L_{0}, L)$, which proves the Theorem.
\end{proof}

\section{Two-space scattering}\label{Two-space_scattering}
In this section, we consider a two-space scattering.

First we treat the short-range case.

\begin{Prop}\label{Prop_wave_operator_free_k_short}
Suppose that $k$ is short-range.
Then the wave operators 
$W^{\pm}(H_{k}, H_{0})$ and $W^{\pm}(H_{0}, H_{k})$ exist and are adjoint each other. They are asymptotically complete:
\begin{gather*}
W^{\pm}(H_{k}, H_{0}) \mathcal{H}_{f} = P_{ac}(H_{k}) \mathcal{H}.
\end{gather*}
\end{Prop}
\begin{proof}
Let $E_{P}(\Lambda)$ be the spectral projections of $P$ on $\Lambda$ with $\Lambda \Subset \mathbb{R}$. We decompose the perturbation term with identifier $E_{P}(\Lambda)$ as follows:
\begin{gather*}
H_{k}E_{P}(\Lambda)-E_{P}(\Lambda)H_{0} = \sqrt{k} P E_{P}(\Lambda) \sqrt{k}.
\end{gather*}
The limiting absorption principle implies that $\sqrt{k}$ is locally $H_{0}$- and $H_{k}$- smooth. $P E_{P}(\Lambda)$ is bounded. The smooth perturbation theory of Kato implies that the wave operators\\
$W^{\pm}(H_{k}, H_{0}; E_{P}(\Lambda))$ and $W^{\pm}(H_{0}, H_{k}; E_{P}(\Lambda))$ exist and are adjoint each other.

Since $P$ commutes with $H_{0}$ and $H_{k}$,
\begin{gather*}
W^{\pm}(H_{k}, H_{0}; E_{P}(\Lambda)) = W^{\pm}(H_{k}, H_{0}) E_{P}(\Lambda), \\
W^{\pm}(H_{0}, H_{k}; E_{P}(\Lambda)) = W^{\pm}(H_{0}, H_{k}) E_{P}(\Lambda).
\end{gather*}
Hence $W^{\pm}(H_{k}, H_{0})$ and $W^{\pm}(H_{0}, H_{k})$ exist and are adoint each other.
\end{proof}

\begin{Prop}\label{Prop_wave_operator_two_space_gomi}
Suppose that $k$ is short-range or long-range.
Then the wave operators \\
$W^{\pm}(L_{0}, H_{k}; J)$ and $W^{\pm}(H_{k}, L_{0}; J^{*})$ exist and are adjoint each other. 
\end{Prop}
\begin{proof}
The perturbation $L_{0}J - J(D_{r}^{2} + k(r) P)$ can be decomposed into a sum of products of first-order differential operator with smooth compactly supported coefficients. Hence we can apply the smooth method of Kato.
\end{proof}
 
Now we obtain the following:

\begin{Thm}\label{Thm_wave_operator_two_space_k_short}
Suppose that $k$ is short-range.
Then the wave operators $W^{\pm}(L_{0}, H_{0}; J)$ and $W^{\pm}(H_{0}, L_{0}; J^{*})$ exist and are adjoint each other. $W^{\pm}(L_{0}, H_{0}; J) \mathcal{H}_{f}^{\mp} = 0$. $W^{\pm}(L_{0}, H_{0}; J)$  and $W^{\pm}(H_{0}, L_{0}; J^{*})$ are isometric on $\mathcal{H}_{f}^{\pm}$ and $P_{ac}(L_{0}) \mathcal{H}$, respectively, and the asymptotic completeness
\begin{gather*}
W^{\pm}(L_{0}, H_{0}; J) \mathcal{H}_{f}^{\pm} = P_{ac}(L_{0}) \mathcal{H}
\end{gather*}
holds.
\end{Thm}
\begin{proof}
It follows from Proposition \ref{Prop_wave_operator_free_k_short} and Proposition \ref{Prop_wave_operator_two_space_gomi} that the wave operators $W^{\pm}(L_{0}, H_{0}; J)$ and $W^{\pm}(H_{0}, L_{0}; J^{*})$ exist and are adjoint each other. 

For $u \in \mathcal{H}_{f}^{\pm}$, 
\begin{gather*}
\lim_{t \to  \pm \infty}\| J e^{-i t H_{0}} u \| = \| u \|, \\
\lim_{t \to  \mp \infty}\| J e^{-i t H_{0}} u \| = 0.
\end{gather*}
Hence $W^{\pm}(L_{0}, H_{0}; J) \mathcal{H}_{f}^{\mp} = 0$, and $W^{\pm}(L_{0}, H_{0}; J)$ is isometric on $\mathcal{H}_{f}^{\pm}$.

To show the isometricity of $W^{\pm}(H_{0}, L_{0}; J^{*})$, it is enough to check that 
\begin{gather*}
\lim_{t \to \pm \infty} \| ( 1 - \chi ) e^{-i t L_{0}} u\| = 0
\end{gather*}
for $u \in P_{ac}(L_{0}) \mathcal{H}$. This follows from the local $L_{0}$-smoothness of $1 - \chi$.
\end{proof}

Combining Theorem \ref{Thm_wave_operator} and Theorem \ref{Thm_wave_operator_two_space_k_short}, we obtain Theorem \ref{Thm_wave_operator_two_space} by virtue of the chain rule of wave operators.
Conversely, Theorem \ref{Thm_wave_operator_two_space_k_short} and Theorem \ref{Thm_wave_operator_two_space} imply Theorem \ref{Thm_wave_operator}. Theorem \ref{Thm_wave_operator_two_space} is essentially solved in \cite{IN10}. Hence Theorem \ref{Thm_wave_operator} with $k(r) = r^{-2}$ is essentially solved in \cite{IN10}. Our result may be considered as an extention of \cite{IN10}.


In the following of this section, we consider smooth long-range $k$.
We also suppose that the coefficient $a_{1}$ in $E$ is separated into two parts, long-range $\theta$- independent term and short-range term:
\begin{gather}
a_{1} = a_{1}^{L}(r) + a_{1}^{S}(r, \theta) \label{a1_short_long} \\
|\del_{r}^{l} a_{1}^{L}(r)| \leq C_{l} \langle r  \rangle^{-\nu_{a_{1}^{L}} -l}, \nu_{a_{1}^{L}} > 0 \label{a1_long}　\\
|\del_{r}^{l} \del_{\theta}^{\alpha} a_{1}^{S}(r, \theta)| \leq C_{l, \alpha} \langle r  \rangle^{-\nu_{a_{1}^{S}} -l},\nu_{a_{1}^{S}} > 1. \label{a1_short}
\end{gather}
Set
\begin{gather*}
H_{L} = D_{r}(1 + a_{1}^{L}(r))D_{r} + k(r) P.
\end{gather*}
We formulate a long-range scattering theory for the triplet  
$(H_{L}, H_{0}; J^{\pm})$ with modified identifiers $J^{\pm} \in B(\mathcal{H}_{f}, \mathcal{H}_{f})$. Since $P$ commutes with $H_{0}$ and $H_{L}$, it is natural to choose $J^{\pm}$ as
\begin{gather}
J^{\pm} = \int J_{\lambda}^{\pm} dE_{P}(\lambda) \label{J_1}
\end{gather}
where
\begin{gather*}
P = \int \lambda dE_{P}(\lambda)
\end{gather*}
is the spectral decomposition of $P$, and $J_{\lambda}^{\pm}$ are bounded operators $L^{2}(\mathbb{R}) \to L^{2}(\mathbb{R})$. Through this decomposition, the problem reduces to the long-range scattering for the triplet \\
$ \left( H_{L, \lambda}, H_{0, \lambda}; J_{\lambda}^{\pm} \right)$ on the real line, where $H_{L, \lambda} = D_{r}(1+a_{1}^{L})D_{r} + \lambda k(r)$ and $H_{0, \lambda} = D_{r}^{2}$ are self-adjoint operators on $L^{2}(\mathbb{R})$.
We choose $J_{\lambda}^{\pm}$ as a pseudo-differential operator with oscillating symbols
\begin{gather}
J_{\lambda}^{\pm} = \chi_{\lambda}^{\pm}(D_{r}) J(\Phi_{\lambda}^{\pm}, a^{\pm}) \label{J_2}\\
J(\Phi_{\lambda}^{\pm}, a^{\pm}) u(r) =  \frac{1}{(2\pi)^{\frac{1}{2}}} \int_{\mathbb{R}} e^{i r \rho+ i \Phi_{\lambda}^{\pm}(r, \rho)}a^{\pm}(r, \rho) \hat{u}(\rho) d\rho \notag \\
a^{\pm}(r, \rho ) = \eta (r) \psi (\rho^{2}) \sigma^{\pm}(r, \rho ) .\label{J_3}
\end{gather}
Here $\eta \in C^{\infty}(\mathbb{R})$ such that $\eta(r) = 0$ near $r=0$ and $\eta(r) = 1$ for large $|r|$, $\psi \in C_{0}^{\infty}(\mathbb{R}_{+}), \chi_{\lambda}^{\pm} \in C_{0}^{\infty}(\mathbb{R})$ and $\sigma^{\pm} = 1$ if $\pm r \rho > 0$ and $\sigma^{\pm} = 0$ if $\pm r \rho \leq 0$. We seach for a PDO $J_{\lambda}^{\pm}$ such that the perturbation
\begin{gather*}
T_{\lambda}^{\pm} = H_{L, \lambda} J_{\lambda}^{\pm} - J_{\lambda}^{\pm} H_{0, \lambda}
\end{gather*}
admits a factorization into a product of $H_{L, \lambda}$- and $H_{0, \lambda}$- smooth operators. 

Roughly speaking, up to compact terms, $T_{\lambda}^{\pm}$ is also a PDO with symbol
\begin{gather*}
t_{\lambda}^{\pm}(r, \rho) = ((1+a_{1}^{L}(r)) (D_{r} + \rho )^{2} - \rho^{2} + \lambda k(r)) e^{i\Phi_{\lambda}^{\pm}(r, \rho)}a^{\pm}(r, \rho) .
\end{gather*}
Let us compute
\begin{align*}
&e^{-i\Phi_{\lambda}^{\pm}}(r, \rho ) ( (1+a_{1}^{L}(r))(D_{r} + \rho )^{2} - \rho^{2} + \lambda k(r)) e^{i\Phi_{\lambda}^{\pm}}(r, \rho) \\
=& (1+a_{1}^{L}(r))(\nabla \Phi_{\lambda}^{\pm} + \rho)^{2} + \lambda k(r) - \rho^{2}
 - i (1+a_{1}^{L}(r)) \Delta \Phi_{\lambda}^{\pm} .
\end{align*}
We want to find  $\Phi_{\lambda}^{\pm}$  such that
\begin{gather*}
(1+a_{1}^{L}(r))(\nabla \Phi_{\lambda}^{\pm} + \rho)^{2} + \lambda k(r) - \rho^{2}
\end{gather*}
is ``small''. In the case $a_{1}^{L} = 0,$ and $ \nu_{k} >  \frac{1}{2}$, it is enough to set
\begin{gather*}
\Phi_{\lambda}^{\pm}(r, \rho)  = - \frac{1}{2\rho} \int_{0}^{r} \lambda k(s) ds.
\end{gather*}
For general $a_{1}^{L}$ and $\nu_{k} > 0$, we need to apply the method of succesive approximations and to keep $[\nu_{k}^{-1}]$ (the largest integer which does not exceed $\nu_{k}^{-1}$) iterations:

\begin{Lem}\label{Lemma_modifier}
Let  $a_{1}^{L}(r), k(r) \in C^{\infty}(\mathbb{R})$ satisfy the smooth long-range condition:
\begin{gather}
| \del_{r}^{l} a_{1}^{L}(r) | \leq C \langle r \rangle^{-\nu_{a_{1}^{L}}- l} \notag \\
| \del_{r}^{l} k(r) | \leq C \langle r \rangle^{-\nu_{k}- l} 
\end{gather}
with $l \in \mathbb{N}$, and $\nu = \max \{ \nu_{a_{1}^{L}},  \nu_{k} \} > 0$. We assume that $\nu^{-1}$ is not an integer.
Let $\Lambda \Subset \mathbb{R} \setminus \{ 0 \}$.
Then for large enough $R$, there exists a $C^{\infty}$-function $\Phi^{\pm}(r, \rho )$ defined on $(r, \rho) \in \Gamma^{\pm}(R, \Lambda) = \{ (r, \rho) | |r| > R, \rho \in \Lambda, \pm r \rho > 0  \}$ such that
\begin{gather}
| \del_{r}^{l} \del_{\rho}^{k} \Phi^{\pm}(r, \rho ) | \leq C (1 + |r|)^{1- \nu - l} \label{Phi_decay}\\
R[ \Phi^{\pm} ] := (1+a_{1}^{L}) | \nabla \Phi^{\pm} + \rho |^{2} + k(r) - \rho^{2} \notag \\
| \del_{r}^{l} \del_{\rho}^{k} R[ \Phi ](r, \rho ) | \leq C (1 + |r|)^{-1 - \epsilon - l} \notag
\end{gather}
where $\nabla = \del_{r}$ and $\epsilon = \nu([\nu^{-1}]+1) - 1 > 0$.
\end{Lem}
\begin{proof}
We only consider the case $\Phi^{+}$ with $\Lambda \subset \mathbb{R}_{+}$, and abbreviate ``+'' . Other cases are similar to prove.

We fix $R > 0$ large enough such that $|a_{1}^{L}(r)| < \frac{1}{2}$ for $|r| > R$.
Set
\begin{gather*}
\Phi^{(0)} (r, \rho): = 0, \\
\Phi^{(1)} (r, \rho): = -\int_{R}^{r} \frac{k(s) + a_{1}^{L}(s)\rho^{2}}{2(1+a_{1}^{L}(s) \rho)} ds, \\
\Phi^{(N+1)} : = \Phi^{(N)} + \phi^{(N+1)}\\
\phi^{(N+1)}(r, \rho): = -\frac{1}{2\rho}\int_{R}^{r} \left( |\nabla \Phi^{(N)})(s, \rho)|^{2} - |\nabla \Phi^{(N-1)}(s, \rho)|^{2}  \right) ds.
\end{gather*}
with $N \geq 1$.

A simple computation gives 
\begin{align*}
R[\Phi^{(2} ] &= (1+a_{1}^{L})( |\nabla\Phi^{(2)}|^{2} - |\nabla\Phi^{(1)}|^{2} ), \\
R[\Phi^{(N+1)} ] &= (1+a_{1}^{L})( |\nabla\Phi^{(N+1)}|^{2} - |\nabla\Phi^{(N)}|^{2} ) +
R[\Phi^{(N)} ] + 2(1+a_{1}) \langle \nabla \phi^{(N+1)}, \rho \rangle.
\end{align*}
Hence by inducetion we have
\begin{align*}
R[\Phi^{(N+1)} ] &= (1+a_{1}^{L})( |\nabla\Phi^{(N+1)}|^{2} - |\nabla\Phi^{(N)}|^{2} ) .
\end{align*}
We have uniformly for $\rho \in \Lambda$,
\begin{align*}
|\del_{r}^{l}\del_{\rho}^{k}\Phi^{(N)}| & \leq C (1+|r|)^{1-\nu -l},\\
|\del_{r}^{l}\del_{\rho}^{k}\phi^{(N)}| & \leq C (1+|r|)^{1-N\nu -l},\\
|\del_{r}^{l}\del_{\rho}^{k}R[\Phi^{(N)}]| & \leq C (1+|r|)^{-(1+N)\nu -l}.
\end{align*}
It is now sufficient to set $\Phi = \Phi^{([\nu^{-1}])}$.
\end{proof}

From now on, we assume that $\Phi_{\lambda}^{\pm}$ satisfy the conclusions of Lemma \ref{Lemma_modifier} with $k$ replaced by $\lambda k$. We also assume that $\eta(r) = 0$ if $|r| < R$ and $\chi_{\lambda}^{\pm}(\rho) =1$ near $\{ \rho + \nabla_{r}\Phi_{\lambda}^{\pm}(r, \rho) : \rho^{2} \in \text{supp} \psi , |r| > R \} $.
Now we state the existence of modified wave operators:

\begin{Lem}
The wave operators
\begin{gather}
W^{\pm}(H_{L}, H_{0}; J^{\pm}), W^{\pm}(H_{0}, H_{L}; (J^{\pm})^{*} ) \label{WO_pp}
\end{gather}
and
\begin{gather}
W^{\pm}(H_{L}, H_{0}; J^{\mp}), W^{\pm}(H_{0}, H_{L}; (J^{\mp})^{*} ) \label{WO_pm}
\end{gather}
exist. Operators \eqref{WO_pp} as well as \eqref{WO_pm} are adjoint each other.
\end{Lem}
\begin{proof}
It is enough to consider the scattering theory for the triplets $(H_{L, \lambda}, H_{0, \lambda}, J_{\lambda}^{\pm})$.

Set $b = (i^{-1}(\del_{r} a_{1}^{L}) \rho + (1+a_{1}^{L})\rho^{2} +\lambda k(r))\chi_{\lambda}^{\pm}(\rho)$. 
$a$ and $b$ are in $\mathcal{S}^{0}$. 
By Theorem \ref{Thm_PDO_FIO_composition}, there exists $d \in \mathcal{S}^{m_{d}}$ with $m_{d} = 0$ such that
\begin{align*}
H_{L, \lambda}J_{\lambda}^{\pm}
&=(D_{r}(1+a_{1}^{L})D_{r} + \lambda k(r))\chi_{\lambda}^{\pm}(D_{r}) J(\Phi_{\lambda}^{\pm}, a^{\pm}) =
b(x, D_{r})J(\Phi_{\lambda}^{\pm}, a^{\pm})\\
&= J(\Phi_{\lambda}^{\pm}, d)
\end{align*}
and admits the asymptotic expansion
\begin{gather*}
d = \sum_{l \geq 0} \frac{1}{l !}d_{l }, \\
d_{l}(r, \rho) = (\del_{\tau}^{l} D_{s}^{l}p)(0, 0,; r, \rho)
\end{gather*}
where
\begin{gather*}
p(s, \tau; r, \rho) = b(r, \rho + \tau + \delta (r, r + s, \rho ) ) a( r + s, \rho )
\end{gather*}
and
\begin{gather*}
\delta(r, q, \rho ) = \int_{0}^{1} (\nabla_{r}\Phi_{\lambda}^{\pm} ) ( (1-t ) r + t q), \rho) dt .
\end{gather*}
In particular, $d_{l} \in \mathcal{S}^{m_{d}-l} = \mathcal{S}^{-l}$ and
\begin{gather*}
d_{0}(r, \rho ) = b(r, \rho + (\nabla_{r} \Phi_{\lambda}^{\pm} )(r, \rho ) ) a(r, \rho ),\\
d_{1}(r, \rho ) = (\del_{\rho}b)(r, \rho + (\nabla_{r} \Phi_{\lambda}^{\pm} )(r, \rho ) ) (D_{r}a)(r, \rho ) \\+
  \langle \del_{\rho}^{2}b(r, \rho + (\nabla_{r} \Phi_{\lambda}^{\pm} )(r, \rho ) ), \frac{1}{2}(\del_{r} D_{r} \Phi_{\lambda}^{\pm})(r, \rho ) \rangle \ a(r, \rho ).
\end{gather*}
\eqref{Phi_decay} implies that $d_{1} \in \mathcal{S}^{-1-\nu}$ where $\nu = \max \{ \nu_{k}, \nu_{a_{1}}^{L} \}$. Hence $d-d_{0} \in \mathcal{S}^{-1-\nu}$.

Set $c(r, \rho) = \rho^{2}$. Then by Theorem \ref{Thm_PDO_FIO_composition} and Theorem \ref{Thm_FIO_PDO_composition}, there exists $e \in \mathcal{S}^{m_{e}}$ with $m_{e} = 0$ such that
\begin{align*}
J_{\lambda}^{\pm}H_{0, \lambda}
&=\chi_{\lambda}^{\pm}(D_{r}) J(\Phi_{\lambda}^{\pm}, a^{\pm}) (D_{r}^{2})\\
&= J(\Phi_{\lambda}^{\pm}, e)
\end{align*}
and admits the asymptotic expansion
\begin{gather*}
e = \sum_{l \geq 0} \frac{1}{l !}e_{l }, \\
e_{l}(r, \rho) = (\del_{\tau}^{l} D_{s}^{l}q)(0, 0,; r, \rho)
\end{gather*}
where
\begin{gather*}
q(s, \tau; r, \rho) =  a( r, \rho + \tau ) \bar{c}(r + s + \gamma(r, \rho + \tau, \rho) , \rho + \tau )
\end{gather*}
and
\begin{gather*}
\gamma(r, \rho, \sigma ) = \int_{0}^{1} (\nabla_{\rho}\Phi_{\lambda}^{\pm} ) ( r, (1- t)\rho + t \sigma ) dt .
\end{gather*}
In particular, $e_{l} \in \mathcal{S}^{m_{e}-l} = \mathcal{S}^{-l}$ and
\begin{gather*}
e_{0}(r, \rho ) = a( r, \rho) \bar{c}( r +  (\nabla_{\rho} \Phi_{\lambda}^{\pm} )(r, \rho ), \rho),\\
e_{1}(r, \rho ) = (\del_{\rho} a)( r, \rho) (D_{r}\bar{c})(r + (\nabla_{\rho} \Phi_{\lambda}^{\pm} )(r, \rho ), \rho )\\
+ a(r, \rho) \big( (D_{r}\del_{\rho}\bar{c})(r + (\nabla_{\rho} \Phi_{\lambda}^{\pm} )(r, \rho ), \rho)
+ \langle (\nabla_{r}D_{r} \bar{c})(r+ (\nabla_{\rho} \Phi_{\lambda}^{\pm} )(r, \rho ), \rho), \frac{1}{2} \nabla_{\rho}\del_{\rho}\Phi_{\lambda}^{\pm}(r, \rho) \rangle
\big).
\end{gather*}
Since $c(r, \rho) = \rho^{2}$, $e_{1}=0$. Hence $e - e_{1} \in \mathcal{S}^{-2}$.

Now we have $T_{\lambda}^{\pm} = J(\Phi_{\lambda}^{\pm}, d-e)$ where $(d-e) - (d_{0} - e_{0} ) \in \mathcal{S}^{-1-\nu}$ and
\begin{gather*}
(d_{0} - e_{0})(r, \rho)
= \big(  i^{-1}(\del_{r}a_{1}^{L})(r) (\rho + \nabla_{r}\Phi_{\lambda}^{\pm}(r, \rho) ) + R[ \Phi_{\lambda}^{\pm}](r, \rho) \big) a(r, \rho),
\end{gather*}
where
\begin{gather*}
R[ \Phi_{\lambda}^{\pm}](r, \rho) = (1+a_{1}^{L}(r)) | \rho + |\Phi_{\lambda}^{\pm}(r, \rho)|^{2} + \lambda k(r) - \rho^{2}.
\end{gather*}
As in Lemma \ref{Lemma_modifier}, we chose $\Phi_{\lambda}^{\pm}$ so that $R[ \Phi_{\lambda}^{\pm}](r, \rho)  a(r, \rho)\in \mathcal{S}^{-1-\epsilon}$ with some $\epsilon > 0$. Therefore $T_{\lambda}^{\pm} = J(\Phi_{\lambda}^{\pm}, d - e)$ with $d - e \in \mathcal{S}^{-1-\epsilon}$ and hence $\langle r \rangle^{\frac{1+\epsilon}{2}} T_{\lambda}^{\pm} \langle r \rangle^{\frac{1+\epsilon}{2}}$ is bounded. The operator $\langle r \rangle^{-\frac{1+\epsilon}{2}}$ is $H_{0, \lambda}$- and $H_{L, \lambda}$-smooth on any positive bounded interval disjoint from eigenvalues of $H_{L, \lambda}$. So the smooth perturbation theory of Kato yields the Lemma.
\end{proof}

Now we show that these wave operators are isometric on suitable subspaces.

\begin{Lem}\label{Wave_Operators_are_isometric_1}
\begin{gather}
\s-lim_{t \to \pm \infty}((J^{\pm})^{*}J^{\pm} - \psi(H_{0}))e^{-iH_{0}t} = 0\label{isometric_1}\\
\s-lim_{t \to \mp \infty}(J^{\pm})^{*}J^{\pm}e^{-iH_{0}t} = 0\label{isometric_2}.
\end{gather}
In particular, if $\Lambda \Subset \mathbb{R}_{+}$ and $\psi \in C_{0}^{\infty}(\mathbb{R}) $ such that $\psi =1$ on $\Lambda$, then
the wave operators $W^{\pm}(H_{L}, H_{0}; J^{\pm})$ are isometric on the subspace $E_{H_{0}}(\Lambda)\mathcal{H}_{f}$ and $W^{\pm}(H_{L}, H_{0}; J^{\mp}) = 0$ .
\end{Lem}
\begin{proof}
Up to a compact term, $(J_{\lambda}^{\pm})^{*}J_{\lambda}^{\pm}$ is a PDO $Q_{\lambda}^{\pm}$ with symbol
\begin{gather*}
\eta^{2}(r)\psi^{2}(\rho^{2})(\sigma^{\pm})^{2}(r, \rho).
\end{gather*}
If $t \to \mp \infty$, then the stationary point $\rho = \frac{r}{2t}$ of the integral
\begin{gather*}
(Q^{\pm}e^{-iH_{0, \lambda}t}u)(r) = \frac{1}{(2\pi )^{\frac{1}{2}}} \eta^{2}(r)^{2}\int_{\mathbb{R}} e^{ir \rho - i \rho^{2}t} \psi^{2}(\rho^{2})(\sigma^{\pm})^{2}(r, \rho) \hat{u}(\rho) d\rho.
\end{gather*}
does not belong to the support of the function $\sigma^{\pm}$. Therefore supposing $\hat{u} \in C_{0}^{\infty}(\mathbb{R})$ and integrating by parts, we estimate this integral by $C_{N}(1+|r|+|t|)^{-N}$ for an arbitrary $N$. This proves \eqref{isometric_2}. We apply the same argument to the PDO with sumbol $\eta^{2}(r)\psi^{2}(\rho^{2})(\sigma^{\pm})^{2}(r, \rho) - \psi^{2}(\rho^{2})$ to prove \eqref{isometric_1}.
\end{proof}
From now on, fix $\Lambda$ and $\psi$ as in Lemma \ref{Wave_Operators_are_isometric_1}.

\begin{Lem}
The wave operators $W^{\pm}(H_{0}, H_{L}; (J^{\pm})^{*})$ are isometric on $E_{H_{L}}(\Lambda) \mathcal{H}_f$.
\end{Lem}
\begin{proof}
By Lemma \ref{Wave_Operators_are_isometric_1}, $W^{\pm}(H_{0}, H_{L}; (J^{\mp})^{*}) = W^{\pm}(H_{L}, H_{0}; J^{\mp})^{*} = 0$. This implies
\begin{gather}
\lim_{t \to \pm \infty}\|J^{\mp *}e^{-iH_{L}t} u \| = 0 , \ u \in E_{H_{L}}(\Lambda) \mathcal{H}_f. \label{qwe}
\end{gather}
Moreover, $J_{\lambda}^{+}(J_{\lambda}^{+})^{ *} + J_{\lambda}^{-} (J_{\lambda}^{-})^{ *} - \psi^{2}(H_{0, \lambda})$ and $\psi^{2}(H_{0, \lambda}) - \psi^{2}(H_{L, \lambda})$ are compact, and \eqref{qwe} implies that
\begin{gather*}
\lim_{t \to \pm \infty}\|(J^{\pm})^{*}e^{-iH_{L}t} u \| = \| u \| , \ u \in E_{H_{L}}(\Lambda) \mathcal{H}_f.
\end{gather*}
This implies the Lemma.
\end{proof}

\begin{Lem}
The wave operators $W^{\pm}(H_{L}, L_{0}+D_{r}a_{1}^{L}D_{r}; J^{*})$ are isometric on $P_{ac}(L_{0})\mathcal{H}$.
\end{Lem}
\begin{proof}
Use the local $L_{0}+D_{r}a_{1}^{L}D_{r}$-smoothness of $1-\chi$.
\end{proof}

\begin{Lem}
The wave operators $W^{\pm}(L_{0}+D_{r}a_{1}^{L}D_{r}, H_{0}; J J^{\pm})$ are isometric on $E_{\Lambda}(H_{0})\mathcal{H}_{f}^{\pm}$ and $W^{\pm}(L_{0}+D_{r}a_{1}^{L}D_{r}, H_{0}; J J^{\pm})\mathcal{H}_{f}^{\mp} = 0$
\end{Lem}
\begin{proof}
It is enough to show that
\begin{gather}
\s-lim_{t \to \pm \infty} [ (J J^{\pm})^{*} J J^{\pm} - \psi(H_{0}) ] e^{-iH_{0}t}P_{\pm} = 0\label{isometric_3}\\
\s-lim_{t \to \pm \infty} (J J^{\pm})^{*} J J^{\pm} e^{-iH_{0}t}P_{\mp} = 0\label{isometric_4}
\end{gather}
where $P_{\pm} = 0$ are projections onto the subspaces $\mathcal{H}_{f}^{\pm}$.
Again up to a compact term,\\
$(J J^{\pm})^{*} J J^{\pm}P_{\mp}$
is a PDO with symbol
\begin{gather*}
\chi^{2}(r) \eta^{2}(r) \psi^{2}(\rho^{2}) (\sigma^{\pm})^{2}(r, \rho ) 1_{\mathbb{R}_{\mp}}(\rho ) = 0.
\end{gather*}
This implies \eqref{isometric_4}. Similarly, up to a compact term,
$(J J^{\pm})^{*} J J^{\pm}P_{\pm}$
is a PDO with symbol
\begin{gather*}
[(\chi^{2}(r) \eta^{2}(r) (\sigma^{\pm})^{2}(r, \rho ) - 1 ] \psi^{2}(\rho^{2})  1_{\mathbb{R}_{\pm}}(\rho ).
\end{gather*}
We apply the same argument as in Lemma \ref{Wave_Operators_are_isometric_1} to prove \eqref{isometric_3}.
\end{proof}

Combinig these results, we obtain the following theorem.
\begin{Thm}\label{Thm_wave_operators_k_long-range}
Suppose $\nu_{a_{2}} = \nu_{b_{1}} = \nu_{b_{2}} = \nu_{V} >1$, $\nu_{a_{3}} = 1$ and $a_{1}$ can be separated into two parts as in \eqref{a1_short_long} - \eqref{a1_short}.
Suppose $k$ is smooth long-range in the sense of Definition \ref{Def_of_k_SR_LR} and let the operators $J^{\pm}$ be defined by \eqref{J_1}, \eqref{J_2}, and \eqref{J_3} with $\Phi_{\lambda}^{\pm}$ satisfying the properties listed in Lemma \ref{Lemma_modifier} with $k$ replaced by $\lambda k$. We also assume that $\psi(\lambda) = 1$ on $\Lambda \Subset \mathbb{R}_{+}$, $\eta(r) = 0$ if $|r| < R$ for large enough $R$ as is taken in Lemma  \ref{Lemma_modifier}. Then the wave operators $W^{\pm}(L, H_{0}; J J^{\pm})$ and $W^{\pm}(H_{0}, L; (J J^{\pm})^{*})$ exist, are adjoint each other, are isometric on $E_{\Lambda}(H_{0} )\mathcal{H}_{f}^{\pm}$ and $E_{\Lambda}(L)P_{ac}(L)\mathcal{H}$, respectively, $W^{\pm}(L, H_{0}; J J^{\pm})\mathcal{H}_{f}^{\mp} = 0$, and the asymptotic completeness
\begin{gather*}
W^{\pm}(L, H_{0}; J J^{\pm}) E_{\Lambda}(H_{0} )\mathcal{H}_{f}^{\pm} = E_{\Lambda}(L)P_{ac}(L)\mathcal{H}
\end{gather*}
holds.
\end{Thm}

\appendix

\section{Pseudo-differential operators with oscillating symbols}

In this appendix, we describe a class of pseudo-differential operators with oscillating symbols. 

We recall the H\"{o}rmander classes $\mathcal{S}^{m}_{\rho, \delta}$ for $m \in \mathbb{R}, \rho > 0, \delta < 1$.
We set $\mathcal{S}^{m}_{\rho, \delta} = \mathcal{S}^{m}_{\rho, \delta}(\mathbb{R}^{d}\times\mathbb{R}^{d})$ consists of functions $a \in C^{\infty}(\mathbb{R}^{d}\times\mathbb{R}^{d})$ such that, for all multi-indices $\alpha, \beta$, there exist $C_{\alpha, \beta}$ such that
\begin{gather*}
|(\del_{x}^{\alpha} \del_{\xi}^{\beta} a) (x, \xi)| \leq C_{\alpha, \beta} (1 + |x|)^{m-|\alpha|\rho + |\beta| \delta }
\end{gather*}
for all $(x, \xi) \in \mathbb{R}^{d}\times\mathbb{R}^{d}$. The best $C_{\alpha, \beta}$ are the semi-norms of the symbol $a$. We denote $\mathcal{S}^{m} = \mathcal{S}^{m}_{1, 0}$.
We say a symbol $a(x, \xi)$ is compactly supported in the variable $\xi$ if there is a compact set $K \Subset \mathbb{R}^{d}$ such that
\begin{gather*}
a(x, \xi) = 0
\end{gather*}
for all $x \in \mathbb{R}^{d}$ if $\xi \notin K$.
We denote the pseuodo-differential operator (PDO) with symbol $a(x, \xi)$ by $a(x, D)$
\begin{gather*}
(a(x, D)u)(x)= \frac{1}{(2\pi)^{\frac{d}{2}}}\int_{\mathbb{R}^{d}} e^{i \langle x, \xi \rangle}a(x, \xi) \hat{u}(\xi) d\xi
\end{gather*}
where $\hat{u}$ is the Fourier transform of $u$
\begin{gather*}
\hat{u}(\xi) = \frac{1}{(2\pi)^{\frac{d}{2}}}\int_{\mathbb{R}^{d}} e^{- i \langle x, \xi \rangle} u(x) dx.
\end{gather*}
The following is elementary.
\begin{Lem}
Suppose that $a \in \mathcal{S}^{m}$ and  $a$ is compactly supported in the variable $\xi$. Then $a(x, D)\langle x \rangle^{-m}$ is bounded in the space $L^{2}(\mathbb{R}^{d})$ and $a(x, D)\langle x \rangle^{-m^{\prime}}$ is compact if $m^{\prime} > m$.
\end{Lem}

Now we define a class of symbols with oscillating factor. Let $\epsilon >0, m \in \mathbb{R}$,  $\Phi \in \mathcal{S}^{1-\epsilon}$, and $a \in \mathcal{S}^{m}$. We denote classes of symbols of the form
\begin{gather*}
e^{i\Phi(x, \xi)}a(x, \xi)
\end{gather*}
by $C^{m}(\Phi)$. We denote the PDO with symbol $e^{i\Phi}a$ by $J(\Phi, a)$
\begin{gather*}
J(\Phi, a) = (e^{i\Phi}a)(x, D).
\end{gather*}
Clearly $C^{m}(\Phi) \subset \mathcal{S}^{m}_{\epsilon, 1-\epsilon}$ so that $C^{m}(\Phi)$ are good classes if $\epsilon > \frac{1}{2}$. On the other hand, the standard calculus fails for operators from these classes if $\epsilon \leq \frac{1}{2}$. However as is shown in \cite{Ya00_2}, $J(\Phi, a_{1})J(\Phi, a_{2})^{*}$ and $J(\Phi, a_{1})^{*}J(\Phi, a_{2})$ become  usual PDO and admit asymptotic expansions.
\begin{Thm}\label{Thm_basic_calculus_for_FIO}
Suppose that $\Phi \in \mathcal{S}^{1-\epsilon}$ with $\epsilon > 0$, and $a_{j} \in \mathcal{S}^{m_{j}}$ for $j=1, 2$ and some numbers $m_{j}$. Suppose that $a_{j}$ are compactly supported in the variable $\xi$. Then the following holds.
\begin{enumerate}
\item
$G = J(\Phi, a_{1})J(\Phi, a_{2})^{*}$ is a PDO with symbol $g \in \mathcal{S}^{m}$ for $m = m_{1} + m_{2}$ and $g(x, \xi)$ admits the asymptotic expansion 
\begin{gather*}
g = \sum_{|\alpha | \geq 0} \frac{1}{\alpha !}g_{\alpha},\\
g_{\alpha}(x, \xi) = \del_{\xi}^{\alpha}(e^{i \Phi(x, \xi)} a_{1}(x, \xi) D_{x}^{\alpha}(e^{-i \Phi(x, \xi)} \bar{a}_{2}(x, \xi)));
\end{gather*}
in particular, $g_{\alpha} \in \mathcal{S}^{m-|\alpha|\epsilon}$.
\item
$H = J(\Phi, a_{2})^{*}J(\Phi, a_{1})$ is a PDO with symbol $h \in \mathcal{S}^{m}$ for $m = m_{1} + m_{2}$ and $h(x, \xi)$ admits the asymptotic expansion 
\begin{gather*}
h = \sum_{|\alpha | \geq 0} \frac{1}{\alpha !}h_{\alpha }, \\
h_{\alpha}(x, \xi ) = D_{x}^{\alpha}(e^{i \Phi(x, \xi)} a_{1}(x, \xi) \del_{\xi}^{\alpha}(e^{-i \Phi(x, \xi)} \bar{a}_{2}(x, \xi )));
\end{gather*}
in particular, $h_{\alpha} \in \mathcal{S}^{m-|\alpha|\epsilon}$.
\item
$J(\Phi, a_{1})$ is bounded in the space $L^{2}(\mathbb{R}^{d})$ if $m_{1} = 0$ and it is compact if $m_{1} < 0$.
\item
Suppose $m_{1} = m_{2} = 0$. Denote by $A$ the PDO with symbol
\begin{gather*}
a(x, \xi) = a_{1}(x, \xi) \bar{a}_{2}(x, \xi) \in \mathcal{S}^{0}.
\end{gather*}
Then
$J(\Phi, a_{1})J(\Phi, a_{2})^{*} - A$ and $J(\Phi, a_{2})^{*}J(\Phi, a_{1}) - A$ are compact in $L^{2}(\mathbb{R}^{d})$.
\end{enumerate}
\end{Thm}
For the proof of Theorem \ref{Thm_basic_calculus_for_FIO}, we refer Yafaev \cite{Ya00_2}.

Next we consider the product of a $PDO$ with oscillating symbol and a usual pseudo-differential operator. The situation is different whether the pseudo-differential operator is on the left and on the right.

\begin{Thm}\label{Thm_PDO_FIO_composition}
Suppose that $\Phi \in \mathcal{S}^{1- \epsilon}, a \in \mathcal{S}^{m_{a}},$ and $ b \in \mathcal{S}^{m_{b}}$ for $\epsilon > 0$ and some $m_{a}, m_{b} \in \mathbb{R}$. Suppose $a$ and $b$ are compactly supported in the variable $\xi$. Then there exists a symbol $d \in \mathcal{S}^{m_{d}}$ for $m_{d} = m_{a} + m_{b}$ such that $d$ is compactly supported in the variable $\xi$,
\begin{gather*}
b(x, D) J(\Phi, a) = J(\Phi, d),
\end{gather*}
and admits the asymptotic expansion
\begin{gather*}
d = \sum_{|\alpha | \geq 0} \frac{1}{\alpha !}d_{\alpha }, \\
d_{\alpha}(x, \eta) = (\del_{\zeta}^{\alpha} D_{z}^{\alpha}p)(0, 0,; x, \eta)
\end{gather*}
where
\begin{gather*}
p(z, \zeta; x, \eta) = b(x, \eta + \zeta + r(x, x + z, \eta ) ) a( x + z, \eta )
\end{gather*}
and
\begin{gather*}
r(x, y, \eta ) = \int_{0}^{1} (\nabla_{x}\Phi ) ( (1-\tau ) x + \tau y), \eta) d\tau .
\end{gather*}
In particular, $d_{\alpha} \in \mathcal{S}^{m_{d}-|\alpha|}$ and
\begin{gather*}
d_{0}(x, \eta ) = b(a, \eta + (\nabla_{x} \Phi )(x, \eta ) ) a(x, \eta ),\\
d_{\alpha}(x, \eta ) = (\del_{\eta}^{\alpha}b)(a, \eta + (\nabla_{x} \Phi )(x, \eta ) ) (D_{x}^{\alpha}a)(x, \eta ) \\+
  \langle \nabla_{\eta} \del_{\eta}^{\alpha}b(a, \eta + (\nabla_{x} \Phi )(x, \eta ) ), \frac{1}{2}(\nabla_{x} D_{x}^{\alpha} \Phi)(x, \eta ) \rangle a(x, \eta )
\end{gather*}
if $|\alpha| = 1$.
\end{Thm}
\begin{proof}
We compute 
\begin{align*}
& (b(x, D)J(\Phi, a) u)(x) \\
=& (2 \pi)^{- \frac{3 n}{2}} \int e^{i \langle x, \xi \rangle - i \langle y, \xi \rangle - + i \langle y, \eta \rangle + i \Phi(y, \eta)}  b(x, \xi ) a(y, \eta)\hat{u}(\eta) d\eta dy d\xi \\
=& (2 \pi)^{- \frac{3 n}{2}} \int e^{i \langle x, \eta \rangle + i \Phi(x, \eta)} \hat{u}(\eta) \big( \int e^{i \langle x-y, \xi - \eta \rangle + i( \Phi(y, \eta) - \Phi(x, \eta) )} b(x, \xi ) a(y, \eta) dy d\xi \big) d\eta \\
=& (2 \pi)^{- \frac{n}{2}} \int e^{i \langle x, \eta \rangle + i \Phi(x, \eta)} \hat{u}(\eta) d(x, \eta) d\eta
\end{align*}
where
\begin{gather*}
d(x, \eta) = (2 \pi)^{- n} \int e^{i \langle x - y, \xi - \eta \rangle + i( \Phi(y, \eta) - \Phi(x, \eta) )} b(x, \xi ) a(y, \eta)dy d\xi.
\end{gather*}
We set
\begin{gather*}
r(x, y, \eta) = \int_{0}^{1} (\nabla_{x} \Phi ) ( (1 -\tau ) x + \tau y, \eta) d\tau.
\end{gather*}
Then
\begin{gather*}
\Phi(y, \eta) - \Phi(x, \eta) = \langle y - x, r(x, y, \eta ) \rangle.
\end{gather*}
By changing variables, we compute
\begin{align*}
&d(x, \eta)\\
=& (2 \pi)^{- n} \int e^{i \langle x - y, \xi - \eta - r(x, y, \eta) \rangle} b(x, \xi ) a(y, \eta) dy d\xi\\
=& (2 \pi)^{- n} \int e^{i \langle x - y, \tilde{\xi} - \eta \rangle} b(x, \tilde{\xi} + r(x, y, \eta) ) a(y, \eta)  dy d\xi\\
=& (2 \pi)^{- n} \int e^{- i \langle z, \zeta \rangle} b(x, \eta + \zeta + r(x, x + z, \eta) ) a(x + z, \eta) dy d\xi.
\end{align*}
Set 
\begin{gather*}
p(z, \zeta; x, \eta ) = b(x, \eta + \zeta + r(x, x + z, \eta) ) a(x + z, \eta).
\end{gather*}
Then by Taylor's expansion formula, we obtain the folowing:
\begin{gather*}
d(x, \eta) = \sum_{ 0 \leq |\alpha| \leq N-1} \frac{1}{\alpha !} (\del_{\zeta}^{\alpha} D_{\zeta}^{\alpha} p )(0, 0; x, \eta) + p^{(N)}(x, \eta)
\end{gather*}
where
\begin{gather*}
p^{(N)}(x, \eta) = (2 \pi)^{- n} N \sum_{|\alpha| = N} \frac{1}{\alpha !} \int_{0}^{1} (1-t)^{N-1} \int \int (\del_{z}^{\alpha}p) (tz, \zeta; x. \eta)z^{\alpha}e^{-i \langle z, \zeta \rangle} dz d\zeta dt.
\end{gather*}
Set 
\begin{gather*}
R^{(\alpha)}(x, \eta; t) = \int \int (\del_{z}^{\alpha} D_{\zeta}^{\alpha} p)(tz, \zeta; x. \eta)e^{-i \langle z, \zeta \rangle} dz d\zeta.
\end{gather*}
Now it is enough to show that $R^{(\alpha)} \in S^{m_{d} - |\alpha|}$ and the seminorms are bounded uniformly with respect to the variable $t$. This obeys from the following two elementary lemmas.

\begin{Lem}
Fix $C > 0$. If $ |z| \geq C |x|$, then for any $n$, 
\begin{gather*}
|\int (\del_{z}^{\alpha} D_{\zeta}^{\alpha} p)(tz, \zeta; x. \eta)e^{-i \langle z, \zeta \rangle} d\zeta | \leq C \langle z \rangle^{-n}.
\end{gather*}

\begin{Lem}
There exists $C > 0$ such that
\begin{gather*}
|\int \int_{|z| \leq C|x|} (\del_{z}^{\alpha} D_{\zeta}^{\alpha} p)(tz, \zeta; x. \eta)e^{-i \langle z, \zeta \rangle} dz d\zeta| \leq C \langle x \rangle^{m_{d} - |\alpha|}.
\end{gather*}
\end{Lem}
\end{Lem}
By integrating by parts, we can show these lemmas.
\end{proof}

\begin{Thm}\label{Thm_FIO_PDO_composition}
Suppose that $\Phi \in \mathcal{S}^{1- \epsilon}, a \in \mathcal{S}^{m_{a}},$ and $ c \in \mathcal{S}^{m_{c}}$ for $\epsilon > 0$ and some $m_{a}, m_{c} \in \mathbb{R}$. Suppose $a$ is compactly supported in the variable $\xi$. Then there exists a symbol $e \in \mathcal{S}^{m_{e}}$ for $m_{e} = m_{a} + m_{c}$ such that
\begin{gather*}
J(\Phi, a) c(x, D)^{*}= J(\Phi, e),
\end{gather*}
and admits the asymptotic expansion
\begin{gather*}
e = \sum_{|\alpha | \geq 0} \frac{1}{\alpha !}e_{\alpha }, \\
e_{\alpha}(x, \eta) = (\del_{\zeta}^{\alpha} D_{z}^{\alpha}q)(0, 0,; x, \eta)
\end{gather*}
where
\begin{gather*}
q(z, \zeta; x, \eta) =  a( x, \eta + \zeta ) \bar{c}(x+ z + s(x, \eta+\zeta, \eta) , \eta + \zeta )
\end{gather*}
and
\begin{gather*}
s(x, \zeta, \eta ) = \int_{0}^{1} (\nabla_{\xi}\Phi ) ( x, (1- \tau)\eta + \tau \xi ) d\tau .
\end{gather*}
In particular, $e_{\alpha} \in \mathcal{S}^{m_{e}-|\alpha|}$ and
\begin{gather*}
e_{0}(x, \eta ) = a( x, \eta) \bar{c}(x+ (\nabla_{\eta} \Phi )(x, \eta ), \eta),\\
e_{\alpha}(x, \eta ) = (\del_{\eta}^{\alpha} a)( x, \eta) (D_{x}^{\alpha}\bar{c})(x+ (\nabla_{\eta} \Phi )(x, \eta ), \eta )\\
+ a(x, \eta) \big( (D_{x}^{\alpha}\del_{\eta}\bar{c}(x+ (\nabla_{\eta} \Phi )(x, \eta ), \eta)
+ \langle (\nabla_{x}D_{x}^{\alpha} \bar{c})(x+ (\nabla_{\eta} \Phi )(x, \eta ), \eta), \frac{1}{2} \nabla_{\xi}\del_{\xi}^{\alpha}\Phi(x, \eta) \rangle
\big)
\end{gather*}
if $|\alpha| = 1$.
\end{Thm}
Proof is similar to Theorem \ref{Thm_PDO_FIO_composition}.

\end{document}